\documentclass[final]{siamart190516}
\usepackage{amsmath}
\usepackage{mathdots}
\usepackage{float}
\usepackage{lipsum}
\usepackage{amsfonts}
\usepackage{graphicx}
\usepackage{subfigure}
\usepackage{multirow}
\usepackage{epstopdf}
\usepackage{amsopn}
\usepackage{amssymb}
\usepackage{enumitem}
\usepackage{amssymb}
\usepackage{algorithm,algpseudocode}
\renewcommand{\algorithmiccomment}[1]{\hfill $\triangleright$ \emph{#1}}
\ifpdf
\DeclareGraphicsExtensions{.eps,.pdf,.png,.jpg}
\else
\DeclareGraphicsExtensions{.eps}
\fi

\newsiamremark{remark}{Remark}
\newsiamremark{hypothesis}{Hypothesis}
\crefname{hypothesis}{Hypothesis}{Hypotheses}
\newsiamthm{claim}{Claim}
\headers{SuperDC: Stable Superfast Eigendecomposition}{Xiaofeng Ou and Jianlin Xia}

\makeatletter
\newcommand*{\addFileDependency}
[1]{
\typeout{(#1)}
\@addtofilelist{#1}
\IfFileExists{#1}{}{\typeout{No file #1.}}
}
\makeatother

\newcounter{example}
\newenvironment{example}{\refstepcounter{example}\vspace{1ex}
{\sc Example \theexample.}\hspace{0.3em}\parindent=0pt}{\vspace{1ex}}
\begin{document}
\title{SuperDC: Stable superfast divide-and-conquer eigenvalue
decomposition\thanks{The research of Jianlin Xia was supported in part by an
NSF grant DMS-1819166.}}
\author{Xiaofeng Ou\thanks{Department of Mathematics, Purdue University, West
Lafayette, IN 47907 (ou17@purdue.edu, xiaj@purdue.edu).}
\and Jianlin Xia\footnotemark[2]}
\maketitle

\begin{abstract}
For dense Hermitian matrices with small off-diagonal (numerical) ranks and in
a hierarchically semiseparable form, we give a stable divide-and-conquer
eigendecomposition method with nearly linear complexity (called SuperDC) that
significantly improves an earlier basic algorithm in [Vogel, Xia, et al., SIAM
J. Sci. Comput., 38 (2016)]. We incorporate a sequence of key stability
techniques and provide many improvements in the algorithm design. Various
stability risks in the original algorithm are analyzed, including potential
exponential norm growth, cancellations, loss of accuracy with clustered
eigenvalues or intermediate eigenvalues, etc. In the dividing stage, we give a
new structured low-rank update strategy with balancing that eliminates the
exponential norm growth and also minimizes the ranks of low-rank updates. In
the conquering stage with low-rank updated eigenvalue solution, the original
algorithm directly uses the regular fast multipole method (FMM) to accelerate
function evaluations, which has the risks of cancellation, division by zero,
and slow convergence. Here, we design a triangular FMM to avoid cancellation.
Furthermore, when there are clustered intermediate eigenvalues or when updates
to existing eigenvalues are very small, we design a novel local shifting
strategy to integrate FMM accelerations into the solution of shifted secular
equations so as to achieve both the efficiency and the reliability. We also
provide several improvements or clarifications on some structures and
techniques that are missing or unclear in the previous work. The resulting
SuperDC eigensolver has significantly better stability while keeping the
nearly linear complexity for finding the entire eigenvalue decomposition. In a
set of comprehensive tests, SuperDC shows dramatically lower runtime and
storage than the Matlab \textsf{eig} function. The stability benefits are also
confirmed with both analysis and numerical comparisons.

\end{abstract}



\section{Introduction}


In this paper, we consider the full eigenvalue decomposition of $n\times n$
Hermitian matrices $A$ with small off-diagonal ranks or numerical\ ranks. Such
matrices belong to the class of rank-structured matrices. Examples include
banded matrices with finite bandwidth, Toeplitz matrices in Fourier space,
some matrices arising from discretized PDEs and integral equations, some
kernel matrices, etc. The eigenvalue decompositions are very useful for
computations such as matrix function evaluations, discretized linear system
solutions, matrix equation solutions, and quadrature approximations. They are
also very useful for fields such as optimization, imaging, Gaussian processes,
and machine learning. In addition, Hermitian eigendecompositions can be used
to compute SVDs of non-Hermitian matrices.

There are several types of rank-structured forms such as $\mathcal{H}$/$\mathcal{H}^{2}$ matrices \cite{hack2002h2, hack1999}, hierarchical
semiseparable (HSS) matrices \cite{cha06,fasthss},
quasiseparable/semiseparable matrices \cite{cha05,van08}, BLR matrices
\cite{ame15}, and HODLR matrices \cite{amb13}. Examples of eigensolvers for
these rank-structured methods include divide-and-conquer methods
\cite{cha04,eid12,lia16,sus21,hsseig}, QR iterations
\cite{bini2005,cha2007,eidelman2005,van2010}, and bisection
\cite{benner2012,xi2014}. Other methods like in \cite{bin91,gu95} have also
been used in the acceleration of relevant eigenvalue solutions.

Our work here focuses on the divide-and-conquer method for HSS\ matrices (that
may be dense or sparse). The divide-and-conquer method has previously been
well studied for tridiagonal matrices (which may be considered as special
HSS\ forms). See, e.g., \cite{lapack,bun78,cuppen1980,dong87,gu95,ole90}. In
particular, a stable version is given in \cite{gu95}. The algorithms can
compute all the eigenvalues in $O(n^{2})$ flops and can compute the
eigenvectors in $O(n^{3})$ flops. It is also mentioned in \cite{gu95} that it
is possible to accelerate the operations in the divide-and-conquer process via
the fast multipole method (FMM) \cite{gre87} so as to reach nearly linear
complexity. However, this has not actually been done in \cite{gu95} or later
relevant work \cite{cha04,lia16}, until more recently in \cite{hsseig} where a
divide-and-conquer algorithm is designed for HSS matrices without the need of
tridiagonal reductions. For an HSS\ matrix with off-diagonal ranks bounded by
$r$ (which may be a constant or a power of $\log n$), the method in
\cite{hsseig} computes a structured eigendecomposition in $O(r^{2}n\log^{2}n)$
flops with storage $O(rn\log n)$. The method is then said to be
\emph{superfast}.

The work in \cite{hsseig} gives a proof-of-concept study of superfast
eigendecompositions for HSS\ matrices $A$. Yet it does not consider some
crucial stability issues in the HSS divide-and-conquer process, such as the
risks of exponential norm growth and potential cancellations in some function
evaluations. Moreover, it does not incorporate several key stability
measurements that are otherwise used in practical tridiagonal
divide-and-conquer algorithms. In fact, these limitations are due to some
major challenges in combining FMM accelerations with those stability
measurements. More specifically, the limitations are as follows.

\begin{enumerate}
\item During the dividing stage, the diagonal blocks of $A$ (also as
HSS\ blocks) are repeatedly updated along a top-down hierarchical tree
traversal. If some upper-level off-diagonal blocks have large norms, the
updated HSS blocks will have subblocks whose norms grow exponentially in the
hierarchical update. This brings stability risks and may even cause overflow,
as can be seen in one of our test examples later.

\item In the conquering stage, the eigenvalues are solved via modified
Newton's method applied to some secular equations. Relevant function
evaluations are assembled into matrix-vector products so as to apply FMM
accelerations. In practical secular equation solution, a function evaluation
may be split into two (say, for the positive terms and negative terms in a
summation) so as to avoid cancellation and also to employ different
interpolation methods \cite{bun78,li93}. Such splitting depends on individual
eigenvalues, so that the usual FMM\ acceleration cannot apply. (See Section
\ref{subsub:assemble}.) In \cite{hsseig}, the FMM\ is used directly without
such splitting, which gives another stability risk.

\item The eigenvalues and eigenvectors are found through a sequence of
intermediate eigenvalue problems. The FMM\ is used to accelerate multiple
parts of the process. When the eigenvalues of $A$ or any of the intermediate
eigenvalue problems are clustered or when an updated eigenvalue is close to a
previous one, the FMM acceleration applied to the standard secular equation
solution will likely lose accuracy or even encounter division by zero due to
catastrophic cancellation. Furthermore, it also impacts the convergence of the
iterative solution and the orthogonality of the eigenvectors. In practical
tridiagonal divide-and-conquer implementations, the issues are nice resolved
through the solution of some \emph{shifted} secular equations for some
eigenvalue gaps. However, such shifting is eigenvalue dependent and there is
no uniform shift that works for all the eigenvalues. This makes it difficult
to apply FMM\ accelerations. (See Section \ref{subsub:shift} for the details.)
Again, the algorithm in \cite{hsseig} directly applies FMM\ accelerations to
standard secular equations without shifting. This is then potentially
dangerous for practical use.

\item In addition, the algorithm in \cite{hsseig} is presented in a
superficial way and some essential components are missing or unclear. It
especially misses the treatment of closely clustered (intermediate)
eigenvalues and small updates to eigenvalues.
\end{enumerate}

The main purpose of this paper is then to overcome these limitations. That is,
we seek to design a stable and superfast divide-and-conquer eigensolver
(called SuperDC) for $A$ in an HSS\ form so as to find an approximate
eigenvalue decomposition
\begin{equation}
A\approx Q\Lambda Q^{T}, \label{eq:eig}\end{equation}
where, for convenience, $A$ is supposed to be real and symmetric since the
ideas can be immediately extended to the Hermitian case, $\Lambda$ is a
diagonal matrix for the eigenvalues, and $Q$ is for the orthogonal
eigenvectors. Also for convenience, we call the matrix $Q$ an
\emph{eigenmatrix}. As compared with the algorithm in \cite{hsseig}, we give a
sequence of techniques that resolves the stability issues. We also provide
many other improvements in terms of the reliability, efficiency, and certain
analysis. The main significance of the work includes the following.

\begin{enumerate}
\item We analyze why the original hierarchical dividing strategy in
\cite{hsseig} can lead to exponential norm growth or accumulation. We then
provide a stable dividing strategy. A\ balancing technique is designed and
guarantees that the norm growth is well under control. We can further save
later eigenvalue solution costs by appropriately tuning the low-rank updates
so as to minimize the rank of the low-rank update.

\item In the solution of the secular equations, when a function evaluation is
split into two for the stability purpose, we design a triangular FMM\ that can
accommodate the eigenvalue dependence so as to stably accelerate the
matrix-vector multiplication resulting from assembling multiple function evaluations.

\item When shifted secular equations are used to handle clustered intermediate
eigenvalues or small eigenvalue updates, we design a \emph{local shifting}
strategy that makes it feasible to apply FMM\ accelerations. Different types
of FMM matrix blocks are treated differently and the feasibility is
justified.
The local shifting is a subtle yet effective way to integrate shifts into
FMM\ matrices without destroying the FMM structure. The major computations in
the eigenvalue decomposition can then be stably accelerated by the FMM. This
improves not only the accuracy, but also the convergence of secular equation solution.

\item We also provide various other improvements and give more precise
discussions on some important structures and techniques that are unavailable
or unclear in \cite{hsseig}. Examples include the precise structure of the
resulting eigenmatrix, the FMM-accelerated iterative eigenvalue solution, the
user-supplied eigenvalue deflation criterion, and also the tuning of the
low-rank updates.

\item All the stabilization techniques still nicely preserve the nearly linear
complexity. That is, the eigendecomposition complexity is still $O(r^{2}n\log^{2}n)$, with $O(rn\log n)$ storage, in contrast with the $O(n^{3})$
complexity and $O(n^{2})$ storage of the classical tridiagonal
divide-and-conquer eigensolver (not to mention that no extra tridiagonal
reduction is needed for dense HSS matrices).

\item We provide comprehensive numerical tests in terms of different types of
matrices with a SuperDC\ package in Matlab. For modest matrix sizes $n$,
SuperDC already has a significantly lower runtime and storage than the Matlab
\textsf{eig} function while producing nice accuracy. In a Toeplitz example
below with $n=32,768$, SuperDC is already about $136$ times faster than
\textsf{eig} with only about $1/15$ of the memory. We also demonstrate the
benefits of our stability techniques in the numerical tests.
\end{enumerate}

In the remaining sections, we begin in Section \ref{sec:review} with a quick
review of the basic HSS divide-and-conquer eigensolver in \cite{hsseig}. Then
the improved stable structured dividing strategy is discussed in Section
\ref{sec:divide}, followed by the stable structured conquering scheme in
Section \ref{sec:conquer}. Section \ref{sec:tests} gives some comprehensive
numerical experiments to demonstrate the efficiency and accuracy. Then Section
\ref{sec:concl} concludes the paper. A\ list of the major algorithms is given
in the supplementary materials.

Throughout this paper, the following notation is used.

\begin{itemize}
\item Lower-case letters in bold fonts like $\mathbf{u}$ are used to denote vectors.

\item $(A_{ij})_{n\times n}$ means an $n\times n$ matrix with the
$(i,j)$-entry $A_{ij}$. Sometimes, a matrix defined by the evaluation of a
function $\kappa(s,t)$ at points $s_{i}$ in a set $\mathbf{s}$ and $t_{j}$ in
a set $\mathbf{t}$ is written as $(\kappa(s_{i},t_{j}))_{s_{i}\in
\mathbf{s},t_{j}\in\mathbf{t}}$.


\item $\operatorname{diag}(\cdots)$ denotes a (block) diagonal matrix.

\item $\operatorname{rowsize}(A)$ and $\operatorname{colsize}(A)$ mean the row
and column sizes of $A$, respectively.

\item $\mathbf{u}\odot\mathbf{v}$ denotes the entrywise (Hadamard) product of
two vectors $\mathbf{u}$ and $\mathbf{v}$.

\item For a binary tree $\mathcal{T}$, we suppose it is in postordering so
that it has nodes $i=1,2,\ldots,\operatorname{root}(\mathcal{T})$, where
$\operatorname{root}(\mathcal{T})$ is the root.

\item $\operatorname{fl}(x)$ denotes the floating point result of $x$.

\item $\epsilon_{\operatorname{mach}}$ represents the machine epsilon.
\end{itemize}

\section{Review of the basic superfast divide-and-conquer
eigensolver\label{sec:review}}

We first briefly summarize the basic superfast divide-and-conquer eigensolver
in \cite{hsseig}. This will help make the understanding of later sections more
convenient. The eigensolver in \cite{hsseig} is a generalization of the
classical divide-and-conquer method for tridiagonal matrices to HSS matrices.

A symmetric HSS matrix $A$ \cite{fasthss} defined with the aid of a
postordered full binary tree $\mathcal{T}$ called \emph{HSS\ tree} has a
nested structure that looks like
\begin{equation}
D_{p}=\begin{pmatrix}
D_{i} & U_{i}B_{i}U_{j}^{T}\\
U_{j}B_{i}^{T}U_{i}^{T} & D_{j}\end{pmatrix}
, \label{eq:hss}\end{equation}
where $p\in\mathcal{T}$ has child nodes $i$ and $j$, so that $D_{p}$ with
$p=\operatorname{root}(\mathcal{T})$ is the entire HSS\ matrix $A$. Here, the
$U$ matrices are off-diagonal basic matrices and also satisfy a nested
relationship $U_{p}=\begin{pmatrix}
U_{i}R_{i}\\
U_{j}R_{j}\end{pmatrix}
$. The $D_{i},U_{i},B_{i}$ matrices are called \emph{HSS generators}
associated with node $i$. The maximum size of the $B$ generators is usually
referred to as the \emph{HSS\ rank} of $A$. We suppose the root of the HSS
tree $\mathcal{T}$ for $A$ is at level $0$, and the children of a node $i$ at
level $l$ are at level $l+1$.

The superfast divide-and-conquer eigensolver in \cite{hsseig} finds the
eigendecomposition (\ref{eq:eig}) of $A$ through a dividing stage and a
conquering stage as follows.

\subsection{Dividing stage\label{sub:div}}

In the dividing stage in \cite{hsseig}, $A$ and its submatrices are
recursively divided into block-diagonal HSS forms plus low-rank updates.
Starting with $p=\operatorname{root}(\mathcal{T})$, suppose $p$ has children
$i$ and $j$. $A=D_{p}$ in (\ref{eq:hss}) can be written as\begin{equation}
D_{p}=\begin{pmatrix}
D_{i}-U_{i}B_{i}B_{i}^{T}U_{i}^{T} & \\
& D_{j}-U_{j}U_{j}^{T}\end{pmatrix}
+\begin{pmatrix}
U_{i}B_{i}\\
U_{j}\end{pmatrix}
\left(
\begin{array}
[c]{cc}B_{i}^{T}U_{i}^{T} & U_{j}^{T}\end{array}
\right)  . \label{eq:divide}\end{equation}
For notational convenience, we suppose the HSS\ rank of $A$ is $r$ and each
$B$ generator has column size $r$. By letting\begin{equation}
\hat{D}_{i}=D_{i}-U_{i}B_{i}B_{i}^{T}U_{i}^{T},\quad\hat{D}_{j}=D_{j}-U_{j}U_{j}^{T},\quad Z_{p}=\begin{pmatrix}
U_{i}B_{i}\\
U_{j}\end{pmatrix}
, \label{eq:diagupd}\end{equation}
we arrive at\begin{equation}
D_{p}=\operatorname{diag}(\hat{D}_{i},\hat{D}_{j})+Z_{p}Z_{p}^{T}.
\label{eq:div}\end{equation}
Here, the diagonal blocks $D_{i}$ and $D_{j}$ are modified so that a rank-$r$
update $Z_{p}Z_{p}^{T}$ can be used instead of a rank-$2r$ update. Note that
the updates to $D_{i}$ and $D_{j}$ in (\ref{eq:diagupd}) follow different
patterns due to the $B_{i}B_{i}^{T}$ term. In \cite{hsseig}, a term $B_{j}^{T}B_{j}$ appears in the update to $D_{j}$ but not $D_{i}$ and there is no
guidance in \cite{hsseig} on which diagonal to put the $B_{i}B_{i}^{T}$ or
$B_{j}^{T}B_{j}$ term. Here, we put $B_{i}B_{i}^{T}$ in the update to $D_{i}$
for the convenience of presentation. Later in our new method, we will give a
clear strategy for this based on the minimization of the column size of
$Z_{p}$ (informally referred to as the \emph{rank of the low-rank update} for convenience).

During this process, the blocks $\hat{D}_{i}$ and $\hat{D}_{j}$ remain to be
HSS\ forms. In fact, it is shown in \cite{hsseig,fasthss} that any matrix of
the form $D_{i}-U_{i}HU_{i}^{T}$ can preserve the off-diagonal basis matrices
of $D_{i}$. Specifically, the following lemma can be used for generator updates.

\begin{lemma}
\label{lem:genupd}\cite{hsseig} Let $\mathcal{T}_{i}$ be the subtree of the
HSS tree $\mathcal{T}$ that has the node $i$ as the root. Then $D_{i}-U_{i}HU_{i}^{T}$ has HSS\ generators $\tilde{D}_{k},\tilde{U}_{k},\tilde
{R}_{k},\tilde{B}_{k}$ for each node $k\in\mathcal{T}_{i}$ as follows:
\begin{align}
\tilde{U}_{k}  &  =U_{k},\quad\tilde{R}_{k}=R_{k},\nonumber\\
\tilde{B}_{k}  &  =B_{k}-(R_{k}R_{k_{l}}\cdots R_{k_{1}})H(R_{k_{1}}^{T}\cdots
R_{k_{l}}^{T}R_{\tilde{k}}^{T}),\label{eq:genupd}\\
\tilde{D}_{k}  &  =D_{k}-U_{k}(R_{k}R_{k_{l}}\cdots R_{k_{1}})H(R_{k_{1}}^{T}\cdots R_{k_{l}}^{T}R_{k}^{T})U_{k}^{T}\quad\text{for a leaf }k\text{,}\nonumber
\end{align}
where $\tilde{k}$ is the sibling node of $k$ and $k\rightarrow k_{l}\rightarrow\cdots\rightarrow k_{1}\rightarrow i$ is the path connecting $k$ to
$i$. Accordingly, $D_{i}-U_{i}HU_{i}^{T}$ and $D_{i}$ have the same
off-diagonal basis matrices.
\end{lemma}

Thus, the HSS generators of $\hat{D}_{i}$ and $\hat{D}_{j}$ can be
conveniently obtained via the generator update procedure (\ref{eq:genupd}).
Then the dividing process can continue on $\hat{D}_{i}$ and $\hat{D}_{j}$ like
above with $p$ in (\ref{eq:divide}) replaced by $i$ and $j$, respectively.

\subsection{Conquering stage\label{conquerstg}}

Suppose eigenvalue decompositions of the subproblems $\hat{D}_{i}$ and
$\hat{D}_{j}$ in (\ref{eq:diagupd}) have been computed as
\begin{equation}
\hat{D}_{i}=Q_{i}\Lambda_{i}Q_{i}^{T},\,\,\hat{D}_{j}=Q_{j}\Lambda_{j}Q_{j}^{T}. \label{eq:subeig}\end{equation}
Then from (\ref{eq:div}), we have
\begin{equation}
D_{p}=\operatorname{diag}(Q_{i},Q_{j})[\operatorname{diag}(\Lambda_{i},\Lambda_{j})+\hat{Z}_{p}\hat{Z}_{p}^{T}]\operatorname{diag}(Q_{i}^{T},Q_{j}^{T}), \label{eq:dp}\end{equation}
where
\begin{equation}
\hat{Z}_{p}=\operatorname{diag}(Q_{i}^{T},Q_{j}^{T})Z_{p}. \label{eq:zupdate}\end{equation}
Consequently, if we can solve the rank-$r$ update problem
\begin{equation}
\operatorname{diag}(\Lambda_{i},\Lambda_{j})+\hat{Z}_{p}\hat{Z}_{p}^{T}=\hat{Q}_{p}\Lambda_{p}\hat{Q}_{p}^{T}, \label{eq:rankrupd}\end{equation}
then the eigendecomposition of $D_{p}$ can be simply retrieved as
\begin{equation}
D_{p}=Q_{p}\Lambda_{p}Q_{p}^{T},\quad\text{with\quad}Q_{p}=\operatorname{diag}(Q_{i},Q_{j})\hat{Q}_{p}. \label{eq:dpeig}\end{equation}

The main task is then to compute the eigendecomposition of the low-rank update
problem (\ref{eq:rankrupd}). To this end, suppose $\hat{Z}_{p}=(\mathbf{z}_{1},\ldots,\mathbf{z}_{r})$, where $\mathbf{z}_{k}$'s are the columns. Then
(\ref{eq:rankrupd}) can be treated as $r$ rank-$1$ update problems
$\operatorname{diag}(\Lambda_{i},\Lambda_{j})+\sum_{k=1}^{r}\mathbf{z}_{k}\mathbf{z}_{k}^{T}$. A basic component is then to quickly find the
eigenvalue decomposition of a diagonal plus rank-$1$ update problem, assumed
to be of the form:
\begin{equation}
\tilde{\Lambda}+\mathbf{v}\mathbf{v}^{T}=\tilde{Q}\Lambda\tilde{Q}^{T},
\label{eq:rank1upd}\end{equation}
where $\tilde{\Lambda}=\operatorname{diag}(d_{1},\ldots,d_{n})$ with
$d_{1}\leq\cdots\leq d_{n}$, $\mathbf{v}=(v_{1},\ldots,v_{n})^{T}$, $\tilde
{Q}=(\mathbf{\tilde{q}}_{1},\ldots,\mathbf{\tilde{q}}_{n})$, and
$\Lambda=\operatorname{diag}(\lambda_{1},\ldots,\lambda_{n})$.

As in the standard divide-and-conquer eigensolver (see, e.g.,
\cite{lapack,cuppen1980,gu95}), finding $\lambda_{k}$ is essentially to solve
for the roots of the following secular equation \cite{gol73}:
\begin{equation}
f(x)=1+\sum_{k=1}^{n}\frac{v_{k}^{2}}{d_{k}-x}=0. \label{eq:seceq}\end{equation}
Newton iterations with rational interpolations may be used and cost $O(n^{2})$
to find all the $n$ roots. Once $\lambda_{k}$ is found, a corresponding
eigenvector looks like $\mathbf{\tilde{q}}_{k}=(\tilde{\Lambda}-\lambda
_{k}I)^{-1}v$. Typically, such an analytical form is not directly used due to
the stability concern. Instead, a method in \cite{gu95} based on L\"{o}wner's
formula can be used to obtain $\mathbf{\tilde{q}}_{k}$ stably.

It is also mentioned in \cite{gu95} that nearly $O(n)$ complexity may be
achieved by assembling multiple operations into matrix-vector multiplications
that can be accelerated by the FMM. This is first verified in \cite{hsseig},
where the complexity of the algorithm for finding the entire
eigendecomposition is $O(r^{2}n\log^{2}n)$ instead of $O(n^{3})$, with the
eigenmatrix $Q$ in (\ref{eq:eig}) given in a structured form that needs
$O(rn\log n)$ storage instead of $O(n^{2})$. In the following sections, we
give a series of stability measurements to yield a divide-and-conquer
eigensolver that is both superfast and stable.

\section{Stable structured dividing strategy\label{sec:divide}}

In this section, we point out a stability risk in the original dividing method
as given in (\ref{eq:divide})--(\ref{eq:diagupd}) and propose a more stable
dividing strategy. We also optimize the rank of the low-rank update.

The stability risk can be illustrated as follows. Consider $\hat{D}_{i}$ in
(\ref{eq:divide}) which is the result of updating $D_{i}$ in the dividing
process associated with the parent $p$ of $i$. Suppose $i$ has children
$c_{1}$ and $c_{2}$ such that
\begin{equation}
D_{i}=\begin{pmatrix}
D_{c_{1}} & U_{c_{1}}B_{c_{1}}U_{c_{2}}^{T}\\
U_{c_{2}}B_{c_{1}}^TU_{c_{1}}^{T} & D_{c_{2}}\end{pmatrix}
,\quad U_{i}=\begin{pmatrix}
U_{c_{1}}R_{c_{1}}\\
U_{c_{2}}R_{c_{2}}\end{pmatrix}
. \label{eq:du}\end{equation}
Then
\[
\hat{D}_{i}=D_{i}-U_{i}B_{i}B_{i}^{T}U_{i}^{T}=\begin{pmatrix}
\tilde{D}_{c_{1}} & U_{c_{1}}\tilde{B}_{c_{1}}U_{c_{2}}^{T}\\
U_{c_{2}}\tilde{B}_{c_{1}}^{T}U_{c_{1}}^{T} & \tilde{D}_{c_{2}}\end{pmatrix}
,
\]
where
\begin{gather}
\tilde{D}_{c_{1}}=D_{c_{1}}-U_{c_{1}}R_{c_{1}}B_{i}B_{i}^{T}R_{c_{1}}^{T}U_{c_{1}}^{T},\quad\tilde{D}_{c_{2}}=D_{c_{2}}-U_{c_{2}}R_{c_{2}}B_{i}B_{i}^{T}R_{c_{2}}^{T}U_{c_{2}}^{T},\nonumber\\
\tilde{B}_{c_{1}}=B_{c_{1}}-R_{c_{1}}B_{i}B_{i}^{T}R_{c_{2}}^{T}.
\label{eq:bc1}\end{gather}

In HSS\ constructions \cite{fasthss}, to ensure stability of HSS\ algorithms,
the $U$ basis generators are often made to have orthonormal columns
\cite{hssstability,toepls}. Accordingly, the $R$ generators satisfy that $\begin{pmatrix}
R_{c_{1}}\\
R_{c_{2}}\end{pmatrix}
$ also has orthonormal columns. Then each $B$ generator has $2$-norm equal to
its associated off-diagonal block. For example $\Vert B_{i}\Vert_{2}=\Vert
U_{i}B_{i}U_{j}^{T}\Vert_{2}$. Furthermore, $\Vert R_{c_{1}}\Vert_{2}\leq1$,
$\Vert R_{c_{2}}\Vert_{2}\leq1$, and (\ref{eq:bc1}) means\begin{equation}
\Vert\tilde{B}_{c_{1}}\Vert_{2}\leq\Vert B_{c_{1}}\Vert_{2}+\Vert B_{i}\Vert_{2}^{2}. \label{eq:normbc1}\end{equation}
If the off-diagonal block $U_{i}B_{i}U_{j}^{T}$ has a large norm, $\Vert
\tilde{B}_{c_{1}}\Vert_{2}$ can potentially be much larger than $\Vert
B_{c_{1}}\Vert_{2}$. We can similarly observe the norm growth with the updated
$D$ generators. This causes norm accumulations of lower-level diagonal and
off-diagonal blocks. Moreover, when the dividing process proceeds on
$\tilde{D}_{c_{1}}$, the norms of the updated $B,D$ generators at lower levels
can grow exponentially.

\begin{proposition}
\label{prop:b0}Suppose the $U_{k}$ generator of $A$ associated with each node
$k$ of $\mathcal{T}$ with $k\neq$ $\operatorname{root}(\mathcal{T})$ has
orthonormal columns and all the original $B_{k}$ generators satisfy $\Vert
B_{k}\Vert_{2}\leq\beta$ with $\beta\gg1$. Also suppose the leaves of
$\mathcal{T}$ are at level $l_{\max}\leq\log_{2}n$. When the original dividing
process in Section \ref{sub:div} proceeds from $\operatorname{root}(\mathcal{T})$ to a nonleaf node $i$, immediately after finishing the dividing
process associated with node $i$,

\begin{itemize}
\item with $i$ at level $l\leq l_{\max}-2$, the updated $B_{k}$ generator
(denoted $\tilde{B}_{k}$) associated with any descendant $k$ of $i$ satisfies\begin{equation}
\Vert\tilde{B}_{k}\Vert_{2}=O(\beta^{2^{l}})=O(\beta^{n/4}); \label{eq:normb1}\end{equation}

\item with $i$ at level $l\leq l_{\max}-1$, the updated $D_{k}$ generator
(denoted $\tilde{D}_{k}$) associated with any leaf descendant $k$ of $i$
satisfies\begin{equation}
\quad\Vert\tilde{D}_{k}\Vert_{2}=\Vert D_{k}\Vert_{2}+O(\beta^{2^{l}})=\Vert
D_{k}\Vert_{2}+O(\beta^{n/2}). \label{eq:normd1}\end{equation}

\end{itemize}
\end{proposition}

\begin{proof}
Following the update formulas in Lemma \ref{lem:genupd}, we just need to show
the norm bound for $\Vert\tilde{B}_{k}\Vert_{2}$. The bound for $\Vert
\tilde{D}_{k}\Vert_{2}$ can be shown similarly.

After the dividing process associated with $\operatorname{root}(\mathcal{T})$
is finished, according to (\ref{eq:genupd}), $\tilde{B}_{k}$ associated with
any descendant $k$ of a child $i$ of $\operatorname{root}(\mathcal{T})$ looks
like
\begin{equation}
\tilde{B}_{k}=B_{k}-(R_{k}R_{k_{m-1}}\cdots R_{k_{1}})H_{i}(R_{k_{1}}^{T}\cdots R_{k_{m-1}}^{T}R_{\tilde{k}}^{T}), \label{eq:bk1}\end{equation}
where $H_{i}=B_{i}B_{i}^{T}$ if $i$ is the left child of $\operatorname{root}(\mathcal{T})$ or $H_{i}=I$ otherwise, $k$ is supposed to be at level $m$ with
sibling $\tilde{k}$, and $k\rightarrow k_{m-1}\rightarrow\cdots\rightarrow
k_{1}\rightarrow i$ is the path connecting $k$ to $i$ in the HSS tree
$\mathcal{T}$. Clearly, $\Vert H_{i}\Vert_{2}\leq\beta^{2}$. With the
orthogonality condition of the $U$ basis generators, $\begin{pmatrix}
R_{c_{1}}\\
R_{c_{2}}\end{pmatrix}
$ also has orthogonal columns. Then we get\begin{equation}
\Vert\tilde{B}_{k}\Vert_{2}\leq\Vert B_{k}\Vert_{2}+\Vert H_{i}\Vert_{2}\leq\beta+\beta^{2}=O(\beta^{2}). \label{eq:bkn1}\end{equation}

Then in the dividing process associated with node $i$ at level $1$, for a
child $c$ of $i$ (see Figure \ref{fig:treediv} for an illustration), the
generator $\tilde{D}_{c}$ is further updated to
\begin{equation}
\hat{D}_{c}=\tilde{D}_{c}-U_{c}H_{c}U_{c}^{T}, \label{eq:dc1}\end{equation}
where $H_{c}=\tilde{B}_{c}\tilde{B}_{c}^{T}$ if $c$ is the left child of $i$
or $H_{c}=I$ otherwise. We have $\Vert H_{c}\Vert_{2}\leq\Vert\tilde{B}_{c}\Vert_{2}^{2}$ for the first case and $\Vert H_{c}\Vert_{2}=1$ for the
second case. From (\ref{eq:bkn1}), we have $\Vert H_{c}\Vert_{2}\leq(\beta
^{2}+\beta)^{2}$. For any descendant $k$ of $c$ with sibling $\tilde{k}$,
(\ref{eq:dc1}) needs to update the generator $B_{k}$ to
\begin{align}
\tilde{B}_{k}=\ B_{k}  &  -(R_{k}R_{k_{m-1}}\cdots R_{k_{2}}R_{c})H_{i}(R_{c}^{T}R_{k_{2}}^{T}\cdots R_{k_{m-1}}^{T}R_{\tilde{k}}^{T}) \label{eq:bk2}\\
&  -(R_{k}R_{k_{m-1}}\cdots R_{k_{2}})H_{c}(R_{k_{2}}^{T}\cdots R_{k_{m-1}}^{T})R_{\tilde{k}}^{T},\nonumber
\end{align}
where the last term on the right-hand side is due to the update associated
with the dividing of $D_{i}$ like in (\ref{eq:bk1}). Then\begin{equation}
\Vert\tilde{B}_{k}\Vert_{2}\leq\Vert B_{k}\Vert_{2}+\Vert H_{i}\Vert_{2}+\Vert
H_{c}\Vert_{2}\leq\beta+\beta^{2}+(\beta^{2}+\beta)^{2}=O(\beta^{4}).
\label{eq:bkn2}\end{equation}
\begin{figure}[ptbh]
\centering\includegraphics[height=0.8in]{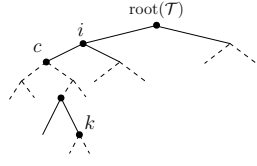}\caption{Nodes involved
in the dividing process.}\label{fig:treediv}\end{figure}

If the dividing process continues to $c$, it is similar to obtain $\Vert
\tilde{B}_{k}\Vert_{2}=O(\beta^{8})$ for any descendant $k$ of a child of $c$.
We can then similarly reach the conclusion on the general pattern of the norm
growth as in (\ref{eq:normb1}). Also, if $i$ is at level $l_{\max}-1$, then
$B_{k}$ associated with a child $k$ of $i$ is not updated, which is why only
$i$ at level $l\leq l_{\max}-2$ contributes to the norm growth of lower level
$B$ generators. This gives $2^{l}\le n/4$.
\end{proof}

This proposition indicates that, during the original hierarchical dividing
process, the updated $B,D$ generators associated with a lower-level node may
potentially have exponential norm accumulation, as long as one of its
ancestors is associated with a $B$ generator with a large norm. This can cause
stability issues or even overflow, as can be seen in the numerical tests later.

To resolve this, we introduce \emph{balancing}/scaling into the updates and
propose a new dividing strategy. That is, we replace the original dividing
method (\ref{eq:divide}) by\begin{align}
D_{p}  &  =\begin{pmatrix}
D_{i}-\frac{1}{\Vert B_{i}\Vert_{2}}U_{i}B_{i}B_{i}^{T}U_{i}^{T} & \\
& D_{j}-\Vert B_{i}\Vert_{2}U_{j}U_{j}^{T}\end{pmatrix}
\label{eq:div1}\\
&  +\begin{pmatrix}
\frac{1}{\sqrt{\Vert B_{i}\Vert_{2}}}U_{i}B_{i}\\
\sqrt{\Vert B_{i}\Vert_{2}}U_{j}\end{pmatrix}
\left(
\begin{array}
[c]{cc}\frac{1}{\sqrt{\Vert B_{i}\Vert_{2}}}B_{i}^{T}U_{i}^{T} & \sqrt{\Vert
B_{i}\Vert_{2}}U_{j}^{T}\end{array}
\right)  .\nonumber
\end{align}
Then we still have (\ref{eq:div}), but with\begin{equation}
\hat{D}_{i}=D_{i}-\frac{1}{\Vert B_{i}\Vert_{2}}U_{i}B_{i}B_{i}^{T}U_{i}^{T},\quad\hat{D}_{j}=D_{j}-\Vert B_{i}\Vert_{2}U_{j}U_{j}^{T},\quad Z_{p}=\begin{pmatrix}
\frac{1}{\sqrt{\Vert B_{i}\Vert_{2}}}U_{i}B_{i}\\
\sqrt{\Vert B_{i}\Vert_{2}}U_{j}\end{pmatrix}
. \label{eq:divnew1}\end{equation}
(For now, $\hat{D}_{i}$ and $\hat{D}_{j}$ still involve the $B_{i}$ terms in
different ways. Slightly later, we will provide a guideline on where to place
these $B_{i}$ terms.)

With this strategy, we can prove that the norms of the updated $B,D$
generators are well controlled.

\begin{proposition}
\label{prop:b}Suppose the same conditions as in Proposition \ref{prop:b0}
hold, except that (\ref{eq:divide}) is replaced by (\ref{eq:div1}) so that
(\ref{eq:diagupd}) is replaced by (\ref{eq:divnew1}). Then (\ref{eq:normb1})
becomes
\begin{equation}
\Vert\tilde{B}_{k}\Vert_{2}\leq2^{l}\beta\leq\frac{n}{4}\beta,
\label{eq:normb2}\end{equation}
and (\ref{eq:normd1}) becomes
\[
\Vert\tilde{D}_{k}\Vert_{2}\leq\Vert D_{k}\Vert_{2}+2^{l}\beta\leq\Vert
D_{k}\Vert_{2}+\frac{n}{2}\beta.
\]

\end{proposition}

\begin{proof}
The proof follows a procedure similar to the proof for Proposition
\ref{prop:b0}. Again, we just show the result for $\Vert\tilde{B}_{k}\Vert
_{2}$. After the dividing process associated with $\operatorname{root}(\mathcal{T})$ is finished, we still have (\ref{eq:bk1}) for any descendant
$k$ of a child $i$ of $\operatorname{root}(\mathcal{T})$, except that
$H_{i}=\frac{B_{i}B_{i}^{T}}{\Vert B_{i}\Vert_{2}}$ if $i$ is the left child
of $\operatorname{root}(\mathcal{T})$ or $H_{i}=\Vert B_{i}\Vert_{2}I$
otherwise. In either case, we have $\Vert H_{i}\Vert_{2}\leq\beta$. Then
(\ref{eq:bkn1}) becomes
\begin{equation}
\Vert\tilde{B}_{k}\Vert_{2}\leq2\beta. \label{eq:bkn3}\end{equation}

Then in the dividing process associated with node $i$ at level $1$, for a
child $c$ of $i$, the generator $\tilde{D}_{c}$ is further updated like in
(\ref{eq:dc1}), except that $H_{c}=\frac{\tilde{B}_{c}\tilde{B}_{c}^{T}}
{\Vert\tilde{B}_{c}\Vert_{2}}$ if $c$ is the left child of $i$ or $H_{c}=\Vert\tilde{B}_{c}\Vert_{2}I$ otherwise. We have $\Vert H_{c}\Vert_{2}
\leq\Vert\tilde{B}_{c}\Vert_{2}$ for both cases. From (\ref{eq:bkn3}), $\Vert
H_{c}\Vert_{2}\leq2\beta$. For any descendant $k$ of $c$, (\ref{eq:dc1}) still
requires the update of the generator $B_{k}$ to $\tilde{B}_{k}$ like in
(\ref{eq:bk2}), except that (\ref{eq:bkn2}) now becomes\[
\Vert\tilde{B}_{k}\Vert_{2}\leq\Vert B_{k}\Vert_{2}+\Vert H_{i}\Vert_{2}+\Vert
H_{c}\Vert_{2}\leq\beta+\beta+2\beta=4\beta.
\]

If the dividing process continues to $c$, it is similar to obtain $\Vert
\tilde{B}_{k}\Vert_{2}\leq8\beta$ for any descendant $k$ of the left child of
$c$. It is clear to observe the norm growth as in (\ref{eq:normb2}) in general.
\end{proof}

Therefore, the norm growth now becomes linear in $n$ and is well controlled,
in contrast with the exponential growth in Proposition \ref{prop:b0}.

Next, we can also optimize the rank of the low-rank update (the number of
columns in $Z_{p}$) in (\ref{eq:div}) and give a guideline to choose how
$\hat{D}_{i}$ and $\hat{D}_{j}$ should involve the $B_{i}$ generator. Note
that in the original dividing method (\ref{eq:divide}) in \cite{hsseig}, the
updates to the two diagonal blocks involve the $B_{i}$ generator in different
ways. No reason is given in \cite{hsseig} to tell why $\hat{D}_{i}$ and
$\hat{D}_{j}$ should involve $B_{i}$ differently.

In fact, in (\ref{eq:divide}) and also (\ref{eq:div1})--(\ref{eq:divnew1}),
the rank of the low-rank update is equal to $\operatorname{colsize}(B_{i})$.
In practice, $B_{i}$ may not be a square matrix. Thus, (\ref{eq:divnew1}) can
be used if $\operatorname{colsize}(B_{i})\leq\operatorname{rowsize}(B_{i})$.
Otherwise, we replace (\ref{eq:divnew1}) by the following:
\begin{equation}
\hat{D}_{i}=D_{i}-\Vert B_{i}\Vert_{2}U_{i}U_{i}^{T},\quad\hat{D}_{j}=D_{j}-\frac{1}{\Vert B_{i}\Vert_{2}}U_{j}B_{i}^{T}B_{i}U_{j}^{T},\quad Z_{p}=\begin{pmatrix}
\sqrt{\Vert B_{i}\Vert_{2}}U_{i}\\
\frac{1}{\sqrt{\Vert B_{i}\Vert_{2}}}U_{j}B_{i}^{T}\end{pmatrix}
, \label{eq:divnew2}\end{equation}
so that (\ref{eq:div}) still holds. In (\ref{eq:divnew2}), the low-rank update
size is now $\operatorname{rowsize}(B_{i})$. With such an optimization
strategy, the rank of the low-rank update is always the smaller of the row and
column sizes of $B_{i}$.

With these new ideas, we arrive at a more stable and efficient dividing stage.
One thing to point out is that the dividing stage needs a step to form $Z_{p}$
like in (\ref{eq:divnew1}) and (\ref{eq:divnew2}). Such a step is not
mentioned in \cite{hsseig}.

\section{Stable structured conquering stage\label{sec:conquer}}

We then discuss the solution of the eigenvalues and eigenvectors in the
conquering stage via the integration of various stability strategies and
FMM\ accelerations. As reviewed in Section \ref{conquerstg}, the key problem
in the conquering stage is to quickly find the eigendecomposition of the
rank-$1$ update problem (\ref{eq:rank1upd}). We show a triangular FMM idea for
accelerating secular equation solution, a local shifting idea for solving
shifted secular equations and constructing structured eigenvectors, the
overall eigendecomposition framework, and the precise eigenmatrix structure.

\subsection{Triangular FMM\ accelerations of secular equation
solution\label{sub:secularFMM}}

For (\ref{eq:rank1upd}), we consider the solution of the secular equation
(\ref{eq:seceq}) for its eigenvalues $\,\lambda_{k},k=1,2,\ldots,n$. Without
loss of generality, suppose the diagonal entries $d_{k}$ of $\tilde{\Lambda}$
are ordered from the smallest to the largest. Also, suppose $d_{k}$ and
$d_{k+1}$ are not too close and each $v_{k}$ is not too small so that
deflation is not needed. Otherwise, deflation in Remark \ref{rem:deflation}
below is applied first.

\subsubsection{Challenge to FMM accelerations of function
evaluations\label{subsub:assemble}}

When modified Newton's method is used to solve for $\lambda_{k}$ as in
practical divide-and-conquer methods, it needs to evaluate $f(x)$ and
$f^{\prime}(x)$ at certain $x_{k}\in(d_{k},d_{k+1})$. The idea in
\cite{cha04,gu95,hsseig} is to assemble the function evaluations for all $k$
together as matrix-vector products and then accelerate them by the FMM. That
is, let
\begin{gather}
\mathbf{f}=\begin{pmatrix}
f(x_{1}) & \cdots & f(x_{n})
\end{pmatrix}
^{T},\quad\mathbf{f}^{\prime}=\begin{pmatrix}
f^{\prime}(x_{1}) & \cdots & f^{\prime}(x_{n})
\end{pmatrix}
^{T},\nonumber\\
\mathbf{v}=\begin{pmatrix}
v_{1} & \cdots & v_{n}\end{pmatrix}
^{T},\quad\mathbf{w}=\mathbf{v}\odot\mathbf{v},\quad\mathbf{e}=\begin{pmatrix}
1 & \cdots & 1
\end{pmatrix}
^{T},\label{eq:w}\\
C=\left(  \frac{1}{d_{j}-x_{i}}\right)  _{n\times n},\quad S=\left(  \frac
{1}{(d_{j}-x_{i})^{2}}\right)  _{n\times n}. \label{eq:cs}\end{gather}
Then\begin{equation}
\mathbf{f}=\mathbf{e}+C\mathbf{w},\quad\mathbf{f}^{\prime}=S\mathbf{w}.
\label{eq:cw}\end{equation}

The vectors $\mathbf{f}$ and $\mathbf{f}^{\prime}$ can be quickly evaluated by
the FMM with the kernel functions $\kappa(s,t)=\frac{1}{s-t}$ and
$\kappa(s,t)=\frac{1}{(s-t)^{2}}$, respectively. A basic idea of the FMM\ for
computing, say, $C\mathbf{w}$ is as follows. Note that $C$ is the evaluation
of the kernel $\kappa(s,t)=\frac{1}{s-t}$ at real points $s\in\{d_{j}\}_{1\leq
j\leq n}$ and $t\in\{x_{i}\}_{1\leq i\leq n}$ that are interlaced:
\begin{equation}
d_{i}<x_{i}<d_{i+1}<x_{i+1},\quad1\leq i\leq n-1. \label{eq:interlace}\end{equation}
The sets $\{x_{i}\}_{1\leq i\leq n}$ and $\{d_{j}\}_{1\leq j\leq n}$ together
are treated as one set and then hierarchically partitioned. This is done by
hierarchically partitioning the interval where all $x_{i}$ and $d_{j}$ are
located. This also naturally leads to a hierarchical partition of both
$\{x_{i}\}_{1\leq i\leq n}$ and $\{d_{j}\}_{1\leq j\leq n}$. Consider two
subsets produced in this partitioning:
\begin{equation}
\mathbf{s}_{x}\subset\{x_{i}\}_{1\leq i\leq n},\quad\mathbf{s}_{d}\subset\{d_{j}\}_{1\leq j\leq n}. \label{eq:subsets}\end{equation}
Use $C_{\mathbf{s}_{x},\mathbf{s}_{d}}=(\kappa(d_{j},x_{i}))_{x_{i}\in\mathbf{s}_{x},d_{j}\in\mathbf{s}_{d}}$ to denote the block of $C$ defined
by $\mathbf{s}_{x}$ and $\mathbf{s}_{d}$, which is often referred to as the
\emph{interaction} between $\mathbf{s}_{x}$ and $\mathbf{s}_{d}$. If
$\mathbf{s}_{x}$ and $\mathbf{s}_{d}$ are well separated (a precise definition
of the separation can be found in \cite{gre87,sun2001}), $C_{\mathbf{s}_{x},\mathbf{s}_{d}}$ is approximated by a low-rank form as
\begin{equation}
C_{\mathbf{s}_{x},\mathbf{s}_{d}}\approx U_{\mathbf{s}_{x}}B_{\mathbf{s}_{x},\mathbf{s}_{d}}V_{\mathbf{s}_{d}}^{T}, \label{eq:cxd}\end{equation}
which can be obtained from a degenerate expansion of $\kappa(s,t)$. For any
desired accuracy, the rank in (\ref{eq:cxd}) is bounded. $\mathbf{s}_{x}$ and
$\mathbf{s}_{d}$ are also said to be \emph{far-field} clusters. If they are
not well separated or are \emph{near-field} clusters, then $C_{\mathbf{s}_{x},\mathbf{s}_{d}}$ is a dense block. The interactions between subsets at
different levels of the hierarchical partition are considered, so that the
$U,V$ basis matrices in (\ref{eq:cxd}) satisfy nested relationships (like in
(\ref{eq:du})). The details can be found in \cite{gre87} and\ are not our
focus here. (Also see \cite{fmm1d} particularly for a stable 1D matrix
version.) The FMM\ essentially produces an \emph{FMM\ matrix }approximation to
$C$ and multiplies it with $\mathbf{w}$. The complexity of each FMM
matrix-vector mutliplication\ is $O(n)$.

In classical practical implementations of secular equation solution methods,
it is preferred to write $f(x)$ as the following form so to avoid cancellation
(see, \cite{bun78}):\[
f(x)=1+\psi_{k}(x)+\phi_{k}(x),
\]
where the splitting depends on $k$ (when $\lambda_{k}\in(d_{k},d_{k+1})$ is to
be found):
\begin{equation}
\label{eq:psiphi1}\psi_{k}(x)=\sum_{j=1}^{k}\frac{v_{j}^{2}}{d_{j}-x},\quad\phi_{k}(x)=\sum_{j=k+1}^{n}\frac{v_{j}^{2}}{d_{j}-x}.
\end{equation}
Due to the interlacing property, all the terms in the sum for $\psi_{k}(x)$
(and $\phi_{k}(x)$) have the same sign. Furthermore, $\psi_{k}$ and $\phi_{k}$
capture the behaviors of $f$ near two poles $d_{k}$ and $d_{k+1}$
respectively. A\ reliable and widely used strategy to solve (\ref{eq:seceq})
is proposed in \cite{li93} based on a modified Newton's method with a hybrid
scheme for rational interpolations. The scheme mixes a middle way method and a
fixed weight method and is implemented in LAPACK \cite{lapack}. In the middle
way method, rational functions $\xi_{k,1}(x)=a_{1}+\frac{b_{1}}{d_{k}-x}$ and
$\xi_{k,2}(x)=a_{2}+\frac{b_{2}}{d_{k+1}-x}$ are decided to interpolate
$\psi_{k}$ and $\phi_{k}$ respectively at $x_{k}\in(d_{k},d_{k+1})$, so that\[
\xi_{k,1}(x_{k})=\psi_{k}(x_{k}),\quad\xi_{k,1}^{\prime}(x_{k})=\psi
_{k}^{\prime}(x_{k}),\quad\xi_{k,2}(x_{k})=\phi_{k}(x_{k}),\quad\xi
_{k,2}^{\prime}(x_{k})=\phi_{k}^{\prime}(x_{k}).
\]
We follow this strategy to find the first $n-1$ roots $\lambda_{1},\lambda
_{2},\ldots,\lambda_{n-1}$. The last root $\lambda_{n}$ has only one pole
$d_{n}$ next to it so a simple rational interpolation is used as in
\cite{lapack, li93}.

In the iterative solution process, it requires to evaluate the functions
$\psi_{k}(x)$, $\phi_{k}(x)$, $\psi_{k}^{\prime}(x)$, and $\phi_{k}^{\prime
}(x)$ at $x_{k}\in(d_{k},d_{k+1})$, $1\leq k\leq n-1$. (Note that even though
the summands in $\psi_{k}^{\prime}(x)$ and $\phi_{k}^{\prime}(x)$ have the
same sign, $\psi_{k}^{\prime}(x)$ and $\phi_{k}^{\prime}(x)$ are used
separately in the rational interpolations by $\xi_{k,1}(x)$ and $\xi_{k,2}(x)$, respectively \cite{li93}.) Since these functions all depend on
individual $k$, the usual FMM\ cannot be applied directly. A\ basic way to
understand this is, the usual FMM handles the evaluation of a kernel
$\kappa(s,t)$ at a fixed set of data points, while here these $k$-dependent
functions need to evaluate the kernel at subsets of the data points that vary
with individual points or $k$.

\subsubsection{Triangular FMM for accelerating the
solution\label{subsub:trifmm}}

To resolve the challenge of applying FMM accelerations to (\ref{eq:psiphi1}),
we let\begin{align}
\boldsymbol{\psi}  &  =\begin{pmatrix}
\psi_{1}(x_{1}) & \cdots & \psi_{n}(x_{n})
\end{pmatrix}
^{T},\quad\boldsymbol{\phi}=\begin{pmatrix}
\phi_{1}(x_{1}) & \cdots & \phi_{n-1}(x_{n-1}) & 0
\end{pmatrix}
^{T},\label{eq:psiphi}\\
\boldsymbol{\psi}^{\prime}  &  =\begin{pmatrix}
\psi_{1}^{\prime}(x_{1}) & \cdots & \psi_{n}^{\prime}(x_{n})
\end{pmatrix}
^{T},\quad\boldsymbol{\phi}^{\prime}=\begin{pmatrix}
\phi_{1}^{\prime}(x_{1}) & \cdots & \phi_{n-1}^{\prime}(x_{n-1}) & 0
\end{pmatrix}
^{T}. \label{eq:psiphi2}\end{align}
The key idea is to write\begin{equation}
\mathbf{f}=\mathbf{e}+\boldsymbol{\psi}+\boldsymbol{\phi}=\mathbf{e}+C_{L}\mathbf{w}+C_{U}\mathbf{w},\quad\mathbf{f}^{\prime}=\boldsymbol{\psi
}^{\prime}+\boldsymbol{\phi}^{\prime}=S_{L}\mathbf{w}+S_{U}\mathbf{w},
\label{eq:flu}\end{equation}
where $\mathbf{e}$ is given in (\ref{eq:w}), $C_{L}$ and $S_{L}$ are the lower
triangular parts of $C$ and $S$, respectively, and $C_{U}$ and $S_{U}$ are the
strictly upper triangular parts of $C$ and $S$, respectively. This suggests
that, to use the FMM, it should be applied to the lower and upper triangular
parts of $C$ and $S$ separately. That is, we need a special \emph{triangular
FMM} that can be used to quickly evaluate $C_{L}\mathbf{w}$, $C_{U}\mathbf{w}$, $S_{L}\mathbf{w}$, $S_{U}\mathbf{w}$.

Without going into too many details, we state some key points in our design of
the triangular FMM in terms of the evaluation of $C_{L}\mathbf{w}$ and
$C_{U}\mathbf{w}$.

\begin{enumerate}
\item During the hierarchical partitioning of (\ref{eq:interlace}) for
generating subsets like in (\ref{eq:subsets}), it is important to guarantee
$d_{k}$ and $x_{k}$ for each same $k$ are respectively assigned to two subsets
$\mathbf{s}_{x}$ and $\mathbf{s}_{d}$ that define a near-field interaction.
This is to make sure $\kappa(d_{k},x_{k})$ appears in a dense block of the FMM
matrix approximation to $C$. Hence, the blocks corresponding to far-field
interactions only consist of entries $\kappa(d_{j},x_{k}),k\neq j$.

\item The partitioning of (\ref{eq:interlace}) should be \emph{adaptive} since
some (intermediate) eigenvalues may cluster together. That is, the interval
where all $x_{i}$ and $d_{j}$ are located may not be uniformly partitioned.

\item The triangular FMM deals with \emph{directional interactions} between
the $x_{i}$ and $d_{j}$ points. For example, for the evaluation of
$C_{L}\mathbf{w}$ in (\ref{eq:flu}) with $\kappa(s,t)=\frac{1}{s-t}$, the
$i$th entry of $C_{L}\mathbf{w}$ is $\sum_{x_{i}>d_{j}}\kappa(d_{j},x_{i})w_{j}$, which corresponds to the interactions between $x_{i}$ and all
$d_{j}$'s on the left of $x_{i}$. For two subsets $\mathbf{s}_{x}$ and
$\mathbf{s}_{d}$ like in (\ref{eq:subsets}), the subblock $(C_{L})_{\mathbf{s}_{x},\mathbf{s}_{d}}$ of $C_{L}$ corresponding to the interaction
between $\mathbf{s}_{d}$ and $\mathbf{s}_{x}$ has the following forms.

\begin{itemize}
\item If $\mathbf{s}_{x}$ and $\mathbf{s}_{d}$ are near-field clusters,
$(C_{L})_{\mathbf{s}_{x},\mathbf{s}_{d}}$ is the lower triangular part of the
dense diagonal block $C_{\mathbf{s}_{x},\mathbf{s}_{d}}$.

\item If $\mathbf{s}_{x}$ and $\mathbf{s}_{d}$ are well separated and
$\mathbf{s}_{x}$ is on the right of $\mathbf{s}_{d}$, $(C_{L})_{\mathbf{s}_{x},\mathbf{s}_{d}}$ is just $C_{\mathbf{s}_{x},\mathbf{s}_{d}}$, so an
approximation in (\ref{eq:cxd}) can be obtained as in the regular FMM.

\item If $\mathbf{s}_{x}$ and $\mathbf{s}_{d}$ are well separated and
$\mathbf{s}_{x}$ is on the left of $\mathbf{s}_{d}$, $(C_{L})_{\mathbf{s}_{x},\mathbf{s}_{d}}$ is a zero block. This can be accommodated by setting
$B_{\mathbf{s}_{x},\mathbf{s}_{d}}=0$ in (\ref{eq:cxd}). In the triangular
FMM, the zero block $(C_{L})_{\mathbf{s}_{x},\mathbf{s}_{d}}$ is skipped in
the matrix-vector multiplication.
\end{itemize}
\end{enumerate}

With the triangular FMM acceleration, it is quick to perform all the function
evaluations in each step of the iterative solution of the secular equation.
The cost in one iteration step for evaluating relevant functions at all
$x_{k}$ simultaneously is $O(n)$.

\subsubsection{Iterative secular equation solution}

During the iterative secular equation solution, let $x_{k}^{(j)}$ be an
approximation to the eigenvalue $\lambda_{k}$ at the iteration step $j$. A
correction $\Delta x_{k}^{(j)}$ is computed so as to update $x_{k}^{(j)}$ as\begin{equation}
x_{k}^{(j+1)}\leftarrow x_{k}^{(j)}+\Delta x_{k}^{(j)}. \label{eq:correction}\end{equation}
(We sometimes write $x_{k}$ instead of $x_{k}^{(j)}$ unless we specifically
discuss the details of the iterations.)

We adopt the stopping criterion from \cite{gu95}:
\begin{equation}
|f(x_{k}^{(j)})|<cn(1+|\psi(x_{k}^{(j)})|+|\phi(x_{k}^{(j)})|)\epsilon
_{\operatorname{mach}}, \label{eq:stopcrit}\end{equation}
where $c$ is a small constant. This stopping criterion can be conveniently
checked after the FMM-accelerated function evaluations, which is an advantage
over a criterion in \cite{li93}. The factor $n$ in (\ref{eq:stopcrit}) might
be loose for extremely large matrices. It is due to the amplification factor
in error propagations of general matrix multiplications. However, the FMM is a
tree-based algorithm where errors propagate along the tree and are amplified
by $O(\log n)$ times instead \cite{hssstability}. Thus for large $n$, $n$ in
(\ref{eq:stopcrit}) may be replaced by $O(\log n)$.

Typically, a very small number of iterations is needed for convergence, just
like the tridiagonal divide-and-conquer algorithm as mentioned in
\cite{dem97}. (In \cite{dem97}, it is pointed out that the LAPACK
divide-and-conquer\ routine reaches full machine precision for each eigenvalue
with only 2 or 3 iterations on average and never more than 7 iterations in
practice.)
With the total number of iterations bounded, the total iterative solution cost
for finding all the eigenvalues (from one secular equation)\ is then $O(n)$.

\begin{remark}
\label{rem:deflation}When $v_{k}$ or the difference $|d_{k}-d_{k+1}|$ is
small, deflation is applied. In practical implementations of the classical
divide-and-conquer eigensolver (see, e.g., \cite{lapack}), the deflation is
performed in a two-step procedure with a tolerance related to $\epsilon
_{\operatorname{mach}}$. Here, we follow a similar procedure, but accept a
user-supplied deflation tolerance $\tau$ to get a more flexible deflation procedure.

\begin{itemize}
\item For $1\leq k \leq n$, $\lambda_{k}$ is deflated if $\vert v_{k} \vert<
\tau$. Without loss of generality, assume $\lambda_{p+1},\ldots,\lambda_{n}$
are deflated, and the remaining eigenvalues are $\lambda_{1},\ldots
,\lambda_{p}$.

\item For $1\leq k \leq p-1$, a Givens rotation is used to deflate
$\lambda_{k}$ if
\[
\vert(d_{k}-d_{k+1}) v_{k} v_{k+1}\vert< (v_{k}^{2}+v_{k+1}^{2}) \tau.
\]

\end{itemize}

The parameter $\tau$ offers the flexibility to control the accuracy of the
eigenvalues. For situations when only modest accuracy is needed, a larger
$\tau$ can be used to save costs. This can sometimes also avoid the need to
deal with situations where $|\lambda_{k}-d_{k}|$ or $|\lambda_{k}-d_{k+1}|$ is
too small.
\end{remark}

\subsection{Local shifting FMM accelerations of shifted secular equation
solution}

When there are clustered eigenvalues or intermediate eigenvalues or when
updates to previous eigenvalues are small, then typically the original secular
equation (\ref{eq:seceq}) is not directly solved. Instead, shifted secular
equations are solved in practical implementations for the purpose of stability
and accuracy, as mentioned in \cite{bun78,dong87,gu95}. However, it is
nontrivial to use the FMM\ to accelerate shifted secular equation solution. In
fact, the paper \cite{gu95} mentions the possibility of FMM\ accelerations for
the original secular equation but does not consider the shifted ones. The
FMM-accelerated algorithm in \cite{hsseig} does not use shifted secular
equations either and thus has stability risks. In this subsection, we discuss
the need for shifts and the challenge to FMM\ accelerations, and moreover,
show how we overcome the challenge through a new strategy that makes it
practical to apply FMM\ accelerations to shifted secular equations. In the
following, we suppose deflation has already been applied.

\subsubsection{Shifted secular equation solution and challenge to FMM
accelerations\label{subsub:shift}}

During the solution for $\lambda_{k}\in(d_{k},d_{k+1})$, if $\lambda_{k}$ is
very close to $d_{k}$ or $d_{k+1}$, a shifted secular equation may be solved
to accurately get the small gap between $\lambda_{k}$ and $d_{k}$ or $d_{k+1}$, after changing the origin to $d_{k}$ or $d_{k+1}$ \cite{bun78,dong87,gu95}.
For example, if $f(\frac{d_{k}+d_{k+1}}{2})\geq0$, then $d_{k}<\lambda_{k}\leq\frac{d_{k}+d_{k+1}}{2}$ and $\lambda_{k}$ is closer to $d_{k}$. The
origin is shifted to $d_{k}$.
Without loss of generality we always assume $\lambda_{k}$ is closer to $d_{k}$
and the shift is $d_{k}$. The original secular equation (\ref{eq:seceq}) can
be written in the following equivalent \emph{shifted secular equation}:
\begin{equation}
g_{k}(y)\equiv f(d_{k}+y)=1+\sum_{j=1}^{n}\frac{v_{j}^{2}}{\delta_{jk}-y}=0,
\label{eq:shiftedeq}\end{equation}
where\begin{equation}
\delta_{jk}= d_{j}-d_{k},\quad j=1,2,\ldots,n. \label{eq:deltajk}\end{equation}
The gap $\eta_{k}\equiv\lambda_{k}-d_{k}$ can be computed accurately by
solving (\ref{eq:shiftedeq}) for $y=\eta_{k}$. We would like to provide some
details on the benefits of this within our context.

One benefit is to avoid catastrophic cancellation or division by zero. When a
high accuracy is desired and a small tolerance $\tau$ is used in the deflation
criterion (Remark \ref{rem:deflation}), shifting is necessary to avoid
catastrophic cancellation or division by zero, similar to the case in standard
divide-and-conquer methods \cite{bun78, gu95}. As discussed in \cite{bun78},
it is preferred to compute $\delta_{ik}-\eta_{k}$ instead of directly from
$d_{i}-\lambda_{k}$, since the former does not suffer from cancellation. To be
more specific, we illustrate this with the following example. In exact
arithmetic, an approximation $x_{k}$ to $\lambda_{k}$ computed in the
iterative solution shall lie strictly between $d_{k}$ and $d_{k+1}$. At each
modified Newton iteration to solve for $\lambda_{k}$ in (\ref{eq:seceq}), it
needs to guarantee $d_{k}<\operatorname{fl}(x_{k})<d_{k+1}$. However, this
might not be satisfied in floating point arithmetic when $x_{k}$ is very close
to $d_{k}$ or
\begin{equation}
|d_{k}-x_{k}|=O(\epsilon_{\operatorname{mach}})\text{ or smaller},
\label{eq:closegap}\end{equation}
which may lead to cancellation when computing $d_{k}-\operatorname{fl}(x_{k})$:
\begin{equation}
\operatorname{fl}(d_{k}-\operatorname{fl}(x_{k}))=o(\epsilon
_{\operatorname{mach}})\quad\text{ or }\quad\operatorname{fl}(d_{k}-\operatorname{fl}(x_{k}))=0 \label{eq:cancellation}\end{equation}
This will induce stability dangers in the numerical solutions of the original
secular function: $\operatorname{fl}\left(  \frac{v_{k}^{2}}{d_{k}-\operatorname{fl}(x_{k})}\right)  $ is either highly inaccurate or becomes
$\infty$.


Note that (\ref{eq:closegap}) and (\ref{eq:cancellation}) are possible even if
deflation has been applied with a tolerance $\tau$ in Remark
\ref{rem:deflation} that is not too small. To see this, suppose $v_{k}=O(\tau)\geq\tau$ and the exact root $\lambda_{k}$ satisfies $|\lambda
_{k}-d_{j}|\gg v_{j}^{2}$ for $j\neq k$. Substituting $\lambda_{k}$ into the
secular equation (\ref{eq:seceq}) to get $\frac{v_{k}^{2}}{d_{k}-\lambda_{k}}=-1+\sum_{j\neq k}^{n}\frac{v_{j}^{2}}{\lambda_{k}-d_{j}}=O(1)$. In this
case, $\lambda_{k}$ shall be very close to $d_{k}$ in the following sense:
\[
|d_{k}-\lambda_{k}|=v_{k}^{2}\cdot O(1)=O(\tau^{2}).
\]
If $\tau=O(\epsilon_{\operatorname{mach}}^{1/2})$ which is not extremely
small, we can have (\ref{eq:closegap}) so that (\ref{eq:cancellation}) may
happen in the modified Newton's method.

Another benefit for solving the shifted equation is the convergence. It is
observed in our tests that dealing with $\eta_{k}$ instead of $\lambda_{k}$
can speed up the convergence of root finding. If $\lambda_{k}$ is solved
directly from (\ref{eq:seceq}), then the approximation $x_{k}^{(j)}$ at
iteration step $j$ is updated as in (\ref{eq:correction}). Suppose
$|\lambda_{k}|=O(1)$ and $|\eta_{k}|=|\lambda_{k}-d_{k}|=O(\epsilon
_{\operatorname{mach}})$. Since $x_{k}^{(j)}$ converges to $\lambda_{k}$ as
$j$ increases, we also have $|x_{k}^{(j)}|=O(1)$ and $|x_{k}^{(j)}-d_{k}|=O(\epsilon_{\operatorname{mach}})$ after some iterations. By modified
Newton's method, the correction $\Delta x_{k}^{(j)}$ approaches $0$ as $j$
increases, which may lead to loss of digits in $x_{k}^{(j+1)}$:
$\operatorname{fl}(x_{k}^{(j+1)})=\operatorname{fl}(x_{k}^{(j)}+\Delta
x_{k}^{(j)})=\operatorname{fl}(x_{k}^{(j)})$. As a result, the iteration
stagnates. On the other hand, if $\eta_{k}$ is solved from the shifted secular
equation, as in \cite{lapack, bun78, dong87}, the update (\ref{eq:correction})
is replaced by
\begin{equation}
y_{k}^{(j+1)}\leftarrow y_{k}^{(j)}+\Delta x_{k}^{(j)}, \label{eq:newupdate}\end{equation}
where $y_{k}^{(j)}=x_{k}^{(j)}-d_{k}$ is an approximation to $\eta_{k}$ at
step $j$ of the iterative solution. Although (\ref{eq:correction}) and
(\ref{eq:newupdate}) are equivalent in exact arithmetic, the latter preserves
a lot more digits of accuracy since $|y_{k}^{(j)}|=O(\epsilon
_{\operatorname{mach}})$.

These discussions illustrate the importance of solving the shifted secular
equation (\ref{eq:shiftedeq}) instead of the original equation (\ref{eq:seceq}). However, in an FMM-accelerated scheme where all $\lambda_{k},k=1,2,\ldots
,n$ are solved simultaneously, it is not convenient to apply the technique of
shifting. This is because the shifts depend on individual eigenvalues and
there is no such a uniform shift that would work for all $\lambda_{k}$'s.

As an example, consider the FMM\ acceleration of the solution of the shifted
equation (\ref{eq:shiftedeq}). Let $y_{k}=x_{k}-d_{k}$ be an approximation to
$\eta_{k}$ during the iterative solution. The evaluations of $g_{k}(y)$ in
(\ref{eq:shiftedeq}) at $y=y_{k}$ for all $k=1,2,\ldots,n$ can be assembled
into the matrix form\begin{gather}
\mathbf{g}=\mathbf{e}+\hat{C}\mathbf{w},\quad\text{with}\label{eq:g}\\
\mathbf{g}=\begin{pmatrix}
g_{1}(y_{1}) & \cdots & g_{n}(y_{n})
\end{pmatrix}
^{T},\quad\hat{C}=\left(  \frac{1}{\delta_{jk}-y_{k}}\right)  _{1\leq k,j\leq
n},\nonumber
\end{gather}
where $\delta_{jk}$ is given in (\ref{eq:deltajk}).

Recall that when the FMM is used to accelerate the matrix-vector product
$C\mathbf{w}$ in (\ref{eq:cw}), it\ relies on the separability of $s$ and $t$
in a degenerate approximation of $\kappa(s,t)=\frac{1}{s-t}$. (Note that in
$\kappa(d_{j},x_{k})$, $x_{k}$ only involves the row index $k$ and $d_{j}$
only involves the column index $j$, so that the separability can be understood
in terms of the row and column indices.) However, to evaluate $\hat
{C}\mathbf{w}$ in (\ref{eq:g}), we have
\begin{equation}
\kappa(d_{j},x_{k})=\kappa(d_{j}-d_{k},x_{k}-d_{k})=\kappa(\delta_{jk},y_{k}).
\label{eq:localshift}\end{equation}
$\delta_{jk}$ involves both the row and column indices, so that the
separability in terms of the row and column indices does not hold. Also, there
is no obvious way of rewriting $\kappa(\delta_{jk},\eta_{k})$ to produce
separability in $j$ and $k$. These make it difficult to apply the
FMM\ acceleration to the solution of the shifted secular equation. (If there
exist such a uniform shift $d_{0}$, then $\kappa(d_{j},x_{k})=\kappa
(d_{j}-d_{0},x_{k}-d_{0})$ and the FMM framework would still apply. However,
the shift $d_{k}$ as above for $\lambda_{k}$ depends on the local behavior of
the secular function in $(d_{k},d_{k+1})$ so such $d_{0}$ does not exist.)

One possible compromise is as follows (as mentioned in our earlier
presentation \cite{cse21}). The FMM-accelerated iterations are applied to
solve the original secular equation (\ref{eq:seceq}) via $K$. In the meantime,
whenever the difference $|x_{k}-d_{k}|$ is too small for a certain eigenvalue
$\lambda_{k}$, switch to solve the shifted equation (\ref{eq:shiftedeq})
without FMM\ accelerations to get $\lambda_{k}$. However, if
(\ref{eq:closegap}) happens very often when a small tolerance $\tau$ is used
for high accuracy or when the problem is not very nice, then the efficiency
will be reduced significantly since every such a case costs extra $O(n)$
flops. Also, when a shift like this is involved, the corresponding eigenvector
needs to be represented in the usual way for the accuracy purpose (instead of
using the structured form as in Section \ref{sub:ev} later). This requires
storages for extra (regular) eigenvectors. Thus, this compromise is not fully satisfactory.

\subsubsection{FMM accelerations with local shifting\label{sub:localshift}}

To resolve the challenge brought by the shifted secular equation, we propose a
somewhat subtle strategy called \emph{local shifting} that makes it feasible
to apply FMM accelerations to solve (\ref{eq:shiftedeq}).

As mentioned in Section \ref{subsub:assemble}, multiple terms involving
$x_{i}-d_{j}$ are assembled into matrices so as to apply FMM\ accelerations.
See, e.g., (\ref{eq:cs}). When $|x_{i}-d_{j}|$ is small, the shifted equation
helps get $x_{i}-d_{j}$ accurately. However, when $i$ is not near $j$ or when
$|i-j|$ is large, $x_{i}-d_{j}$ can actually be computed accurately
\emph{without involving any shift} $d_{k}$ used for computing any eigenvalue
$\lambda_{k}$. To see this, recall that $d_{i}<x_{i}<d_{i+1}$ and also after
deflation with the criterion in Remark \ref{rem:deflation}, we have for all
$i$,
\[
|d_{i}-d_{i+1}|\geq\frac{{v_{i}}^{2}+v_{i+1}^{2}}{v_{i}v_{i+1}}\ge2\tau.
\]
Thus, for $j\neq i,i+1$,
\begin{equation}
|x_{i}-d_{j}|\geq\min(|d_{i}-d_{j}|,|d_{i+1}-d_{j}|)\geq2(|i-j|-1)\tau.
\label{eq:farfield}\end{equation}
Hence, $x_{i}-d_{j}$ can be computed accurately when $|i-j|$ is large.

Following this justification, we have our local shifting strategy with the
following basic ideas.

\begin{enumerate}
\item Use the gap $\eta_{k}$ for each eigenvalue $\lambda_{k}$ locally (in
near-field interactions), which does not interfere with the structures needed
for FMM\ accelerations.

\item It is safe to directly use $\lambda_{k}$ recovered from
\begin{equation}
\lambda_{k}=d_{k}+\eta_{k},\quad k=1,2,\ldots,n, \label{eq:evalnew}\end{equation}
in far-field interactions so as to exploit the rank structure and facilitate FMM\ accelerations.
\end{enumerate}

The major components are as follows.

\begin{itemize}
\item For $k=1,2,\ldots,n$, the shifted secular equations (\ref{eq:shiftedeq})
are solved together for the gaps $\eta_{k}=\lambda_{k}-d_{k}$. An intermediate
gap during the iterative solution looks like $y_{k}=x_{k}-d_{k}$. The relevant
function evaluations in the iterative solutions are assembled into
matrix-vector products like in (\ref{eq:g}).

\item The FMM\ is used to accelerate the resulting matrix-vector products like
$\hat{C}\mathbf{w}$ in (\ref{eq:g}). Suppose two subsets $\mathbf{s}_{x}$ and
$\mathbf{s}_{d}$ like in (\ref{eq:subsets}) are well-separated. As mentioned
above, for $x_{k}\in\mathbf{s}_{x}$ and $d_{j}\in\mathbf{s}_{d}$, $x_{k}$ and
$d_{j}$ are far away from each other and $|k-j|$ is large, so $x_{k}-d_{j}$
can then be computed accurately because of (\ref{eq:farfield}). Thus, we can
recover $x_{k}$ from $d_{k}+y_{k}$ so as to directly exploit the low-rank
structure like in (\ref{eq:cxd}). This is because, say, the far-field
interaction $(\kappa(d_{j},x_{k}))_{x_{k}\in\mathbf{s}_{x},d_{j}\in
\mathbf{s}_{d}}$ (with $\kappa(s,t)=\frac{1}{s-t}$) of $\hat{C}$ is now just a
block of $C$ in (\ref{eq:cs}): $\hat{C}_{\mathbf{s}_{x},\mathbf{s}_{d}}=C_{\mathbf{s}_{x},\mathbf{s}_{d}} $.

\item When two subsets $\mathbf{s}_{x}$ and $\mathbf{s}_{d}$ are not well
separated, the near-field interaction $(\kappa(d_{j},x_{i}))_{x_{i}\in\mathbf{s}_{x},d_{j}\in\mathbf{s}_{d}}$ is kept dense and each entry
$\kappa(d_{j},x_{k})$ can be evaluated accurately in terms of $y_{k}$ and
$\delta_{jk}$ as in (\ref{eq:localshift}). That is, $\hat{C}_{\mathbf{s}_{x},\mathbf{s}_{d}}=\left(  \frac{1}{\delta_{jk}-y_{k}}\right)  _{d_{k}+y_{k}\in\mathbf{s}_{x},d_{k}+\delta_{jk}\in\mathbf{s}_{d}} $. This has no
impact on the structures needed for FMM\ accelerations.

\item These ideas are then combined with the triangular FMM in Section
\ref{subsub:trifmm} so as to stably and quickly perform function evaluations
like (\ref{eq:g}) and solve the shifted secular equations.
\end{itemize}

This local shifting strategy successfully integrates the shifting technique
into the triangular FMM framework without sacrificing performance. It thus
ensures both the efficiency and the stability. We can then quickly and
reliably solve the shifted secular equations as in (\ref{eq:shiftedeq}) via
modified Newton's method to get updates as in (\ref{eq:newupdate}). The
overall complexity to find all the $n$ roots is still $O(n)$. In addition,
since the relevant functions are now evaluated more accurately than with the
method in \cite{hsseig}, the convergence is also improved. (This can be
confirmed from our tests later.) When the iterative solution of the shifted
secular equations converge, we can use the resulting $\eta_{k}$ values to
recover the desired eigenvalues as in (\ref{eq:evalnew}).

The local shifting strategy can also be used to stably apply triangular
FMM\ accelerations to other operations like finding the eigenmatrix. See the
next subsection.

\subsection{Structured eigenvectors via FMM with local shifting}

\label{sub:ev}

With the identified eigenvalues $\lambda_{k}$ in (\ref{eq:evalnew}), the
eigenvectors can be obtained stably as in \cite{gu95}. An eigenvector
corresponding to $\lambda_{k}$ looks like
\begin{equation}
\mathbf{q}_{k}=\left(
\begin{array}
[c]{ccccc}\frac{\hat{v}_{1}}{d_{1}-\lambda_{k}} & \cdots & \frac{\hat{v}_{k}}{d_{k}-\lambda_{k}} & \cdots & \frac{\hat{v}_{n}}{d_{n}-\lambda_{k}}\end{array}
\right)  ^{T}, \label{eq:denseeigvec}\end{equation}
where $\mathbf{\hat{v}}\equiv(\begin{array}
[c]{ccc}\hat{v}_{1} & \cdots & \hat{v}_{n}\end{array}
)^{T}$ is given by L\"{o}wner's formula
\begin{equation}
\hat{v}_{i}=\sqrt{\frac{\prod_{j}(\lambda_{j}-d_{i})}{\prod_{j\neq i}(d_{j}-d_{i})}},\quad i=1,2,\ldots,n. \label{eq:lowner}\end{equation}
To quickly form $\mathbf{\hat{v}}$, the usual FMM acceleration would look like
the following \cite{gu95}. Rewrite (\ref{eq:lowner}) as
\begin{equation}
\log\hat{v}_{i}=\frac{1}{2}\sum_{j=1}^{n}\log(|d_{i}-\lambda_{j}|)-\frac{1}{2}\sum_{j=1,j\neq i}^{n}\log|d_{i}-d_{j}|. \label{eq:lowner2}\end{equation}
Now let $G_{1}=\left(  \log|d_{i}-\lambda_{j}|\right)  _{n\times n}$,
$G_{2}=\left(  \log|d_{i}-d_{j}|\right)  _{n\times n}$, where the diagonals of
$G_{2}$ are set to be zero. Then
\begin{equation}
\log\mathbf{\hat{v}}=\frac{1}{2}(G_{1}\mathbf{e}-G_{2}\mathbf{e}).
\label{eq:logv}\end{equation}
$G_{1}\mathbf{e}$ and $G_{2}\mathbf{e}$ can thus be quickly evaluated by the
FMM with the kernel $\log|s-t|$.

As in \cite{gu95,hsseig}, the eigenvectors are often normalized to form an
orthogonal matrix
\begin{equation}
\hat{Q}=\left(  \frac{\hat{v}_{i}b_{j}}{d_{i}-\lambda_{j}}\right)  _{n\times
n}, \label{eq:eigvecs}\end{equation}
where\begin{equation}
\mathbf{b}\equiv(\begin{array}
[c]{ccc}b_{1} & \cdots & b_{n}\end{array}
)^{T},\quad\text{with\quad}b_{j}=\left(  \sum_{i=1}^{n}\frac{\hat{v}_{i}^{2}}{(d_{i}-\lambda_{j})^{2}}\right)  ^{-1/2}. \label{eq:b}\end{equation}
Again, the vector $\mathbf{b}$ can be quickly obtained via the FMM with the
kernel $\kappa(s,t)=\frac1{(s-t)^{2}}$. $\hat{Q}$ is a Cauchy-like matrix
which gives a structured form of the eigenvectors. The FMM\ with the kernel
$\kappa(s,t)=\frac1{s-t}$ can be used to quickly multiply $\hat{Q}$ to a vector.

Again, with the same reasons as before, all the stability measurements make it
challenging to apply the usual FMM\ to accelerate operations like the
evaluations of $\log\mathbf{v}$ in (\ref{eq:logv}) and $\mathbf{b}$ in
(\ref{eq:b}) and the application of $\hat{Q}$ to a vector. On the other hand,
just like the discussions in Section \ref{sub:localshift}, the local shifting
strategy still applies with appropriate kernels $\kappa(s,t)$.

Thus, instead of directly applying the usual FMM\ accelerations in
\cite{hsseig}, we use triangular FMM accelerations with local shifting. For
example, with the gaps $\eta_{k}$ from the shifted secular equation solution,
it is preferred to use $\delta_{ik}-\eta_{k}$ in place of $d_{i}-\lambda_{k}$
in the computation of some entries of $\mathbf{q}_{k}$ for accuracy purpose
\cite{lapack,bun78,dong87,gu95} when $d_{i}$ and $\lambda_{k}$ are very close.
Note that, with $\delta_{jk}$ in (\ref{eq:deltajk}), (\ref{eq:denseeigvec})
can be written as
\begin{equation}
\mathbf{q}_{k}=\left(
\begin{array}
[c]{ccccc}\frac{\hat{v}_{1}}{\delta_{1k}-\eta_{k}} & \cdots & \frac{\hat{v}_{k}}{-\eta_{k}} & \cdots & \frac{\hat{v}_{n}}{\delta_{nk}-\eta_{k}}\end{array}
\right)  ^{T}. \label{eq:qk}\end{equation}
Then when an entry of $\mathbf{q}_{k}$ belongs to a near-field block of
$\hat{Q}$, its representation in (\ref{eq:qk}) is used. Otherwise, we use its
form in (\ref{eq:denseeigvec}). This preserves the far-field rank structure
and makes the local shifting idea go through.

Thus, triangular FMM accelerations with local shifting can be used to reliably
represent and apply $\hat{Q}$. Note that\begin{equation}
\hat{Q}=\operatorname{diag}(\mathbf{\hat{v}})\left(  \frac{1}{d_{i}-\lambda_{j}}\right)  _{n\times n}\operatorname{diag}(\mathbf{b}).
\label{eq:qscale}\end{equation}
We then store the following five vectors so as to stably retrieve $\hat{Q}$:\begin{equation}
\mathbf{\hat{v}},\ \mathbf{b},\ \mathbf{d}\equiv(\begin{array}
[c]{ccc}d_{1} & \cdots & d_{n}\end{array}
)^{T},\ \boldsymbol{\lambda}\equiv(\begin{array}
[c]{ccc}\lambda_{1} & \cdots & \lambda_{n}\end{array}
)^{T},\ \boldsymbol{\eta}\equiv(\begin{array}
[c]{ccc}\eta_{1} & \cdots & \eta_{n}\end{array}
)^{T}. \label{eq:vec}\end{equation}
Here, we have the storage of one more vector $\boldsymbol{\eta}$ than that in
\cite{hsseig}. This only slightly increase the storage, but the stability is
substantially enhanced.

\subsection{Overall eigendecomposition and structure of the eigenmatrix $Q$}

The overall conquering framework is similar to \cite{hsseig}, but with all the
new stability measurements integrated. Also, the structure of the eigenmatrix
$Q$ is only briefly mentioned in \cite{hsseig} in a vague way. Here, we would
like to give a precise description of $Q$ resulting from the conquering
process and point out an essential component that is missing from
\cite{hsseig}.

The conquering process is performed following the postordered traversal of the
HSS tree $\mathcal{T}$ of $A$, where at each node $i\in\mathcal{T}$, a local
eigenproblem is solved. For a leaf node $i$, suppose $\hat{D}_{i}$ is the
(small) diagonal generator resulting from the overall dividing process.
Compute the dense eigenproblem $\hat{D}_{i}=Q_{i}\Lambda_{i}Q_{i}^{T}$. Then
$Q_{i}$ is a \emph{local eigenmatrix} associated with $i$.

For a non-leaf node $p$ with children $i$ and $j$, the local eigenproblem is
to find an eigendecomposition like in (\ref{eq:dpeig}) based on
(\ref{eq:subeig}) and (\ref{eq:dp}). However, unlike (\ref{eq:rankrupd}) where
a diagonal plus low-rank update eigendecomposition is computed, it is
\emph{necessary to reorder} the diagonal entries of $\operatorname{diag}(\Lambda_{i},\Lambda_{j})$ due to the need to explore structures in the
FMM\ accelerations that rely on the locations of the eigenvalues. Let $P_{p}$
represent a sequence of permutations for deflation and for ordering the
diagonal entries of $\operatorname{diag}(\Lambda_{i},\Lambda_{j})$ from the
smallest to the largest. (Note that the need for $P_{p}$ is not clearly
mentioned in \cite{hsseig}.) Also let the eigendecomposition of the
\emph{permuted} diagonal plus low-rank update problem be\begin{equation}
P_{p}[\operatorname{diag}(\Lambda_{i},\Lambda_{j})+\hat{Z}_{p}\hat{Z}_{p}^{T}]P_{p}^{T}=\hat{Q}_{p}\Lambda_{p}\hat{Q}_{p}^{T}, \label{eq:local}\end{equation}
where $\hat{Z}_{p}$ is given in (\ref{eq:zupdate}). Write $D_{p}$ in
(\ref{eq:dp}) as $\hat{D}_{p}$ since $D_{p}$ is likely updated after the
multilevel dividing process. Then we have the following eigendecomposition:
\begin{equation}
\hat{D}_{p}=Q_{p}\Lambda_{p}Q_{p}^{T},\quad\text{with\quad}Q_{p}=\operatorname{diag}(Q_{i},Q_{j})P_{p}^{T}\hat{Q}_{p}, \label{eq:dphateig}\end{equation}
where $Q_{i}$ and $Q_{j}$ are eigenmatrices of $\hat{D}_{i}$ and $\hat{D}_{j}$
obtained in steps $i$ and $j$, respectively. Then the conquering process
proceeds similarly.

Here for convenience, we say $Q_{p}$ is a \emph{local eigenmatrix} and
$\hat{Q}_{p}$ is an \emph{intermediate eigenmatrix}. The difference between
the two is that a local eigenmatrix is an eigenmatrix of a local HSS\ block
while the latter is an eigenmatrix of a diagonal plus low-rank update problem.
A\ local eigenmatrix is formed by a sequence of intermediate ones. Since
$\hat{Q}_{p}\Lambda_{p}\hat{Q}_{p}^{T}$ in (\ref{eq:local}) is obtained by
solving $r$ consecutive rank-$1$ update eigenproblems, the intermediate
eigenmatrix $\hat{Q}_{p}$ is the product of $r$ Cauchy-like matrices like in
(\ref{eq:eigvecs}). Of course, when FMM\ accelerations and deflation are
applied, the eigendecomposition is approximate.

Then the overall eigenmatrix $Q$ is given in terms of all the intermediate
eigenmatrices, organized with the aid of the tree $\mathcal{T}$. Its precise
form is missing from \cite{hsseig}. Here, we give an accurate way to
understand its structure as follows.

\begin{lemma}
\label{lem:q}Assemble all the intermediate eigenmatrices and permutation
matrices corresponding to the nodes at a level $l$ of $\mathcal{T}$ as
\begin{equation}
Q^{(l)}=\operatorname{diag}(\hat{Q}_{i},\quad i\text{: at level }l\text{ of
}\mathcal{T}),\quad P^{(l)}=\operatorname{diag}(P_{i},\quad i\text{: at level
}l\text{ of }\mathcal{T}). \label{eq:ql}\end{equation}
Then the final eigenmatrix $Q$ has the form (illustrated in Figure
\ref{fig:eigmat})
\begin{equation}
Q=Q^{(l_{\max})}{\textstyle\prod\limits_{l=l_{\max}-1}^{0}}(P^{(l)}Q^{(l)}),
\label{eq:eigmat}\end{equation}
where level $l_{\max}$ is the leaf level of $\mathcal{T}$ and
$\operatorname{root}(\mathcal{T})$ is at level $0$. In addition, $Q$ also
corresponds to (\ref{eq:dphateig}) with $p$ set to be $\operatorname{root}(\mathcal{T})$.
\end{lemma}

\begin{figure}[ptbh]
\centering\includegraphics[height=0.8in]{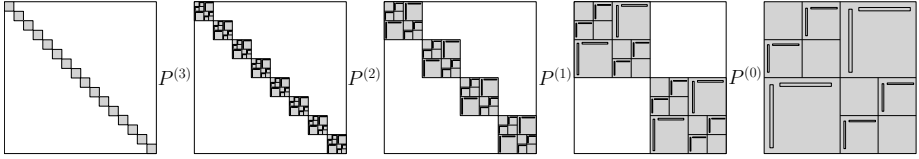}\caption{Illustration of
the structure of the eigenmatrix $Q$, where $l_{\max}=4$ and each structured
diagonal block (marked in gray) is for an intermediate eigenmatrix $\hat
{Q}_{p}$ associated with a nonleaf node $p$.}\label{fig:eigmat}\end{figure}

Thus, $Q$ can be understood in terms of either (\ref{eq:eigmat}) or the local
eigenmatrices. Lemma \ref{lem:q} gives an efficient way to apply $Q$ or
$Q^{T}$ to a vector, where the triangular FMM with local shifting is again
used to multiply the intermediate eigenmatrices with vectors. Note that with a
very similar procedure, a local eigenmatrix $Q_{i}$ or its transpose can be
conveniently applied to a vector. Such an application process is used to multiply
the local eigenmatrices $Q_{i}^{T}$ and $Q_{j}^{T}$ to $Z_{p}$ as in
(\ref{eq:zupdate}) so as to quickly form $\hat{Z}_{p}$ used in (\ref{eq:local}).

In addition, as mentioned in \cite{hsseig}, each intermediate eigenmatrix
$\hat{Q}_{i}$ has small off-diagonal numerical ranks. An off-diagonal
numerical rank result given in \cite{hsseig} is in terms of entrywise
approximations. The overall eigenmatrix $Q$ itself does not necessarily have a
small off-diagonal numerical rank, so a remark in \cite{hsseig} is not
precise. In fact, a precise off-diagonal numerical rank bound for $\hat{Q}_{i}$ in terms of singular value truncation can be shown based on the studies
in \cite{mhs}. Then, if the matrices $P^{(l)}$ are dropped from $Q$ in
(\ref{eq:eigmat}), the resulting matrix has small off-diagonal numerical ranks.

The main algorithms used in SuperDC are shown in the supplementary materials. When $A$ is given in terms of an HSS\ form with HSS rank $r$,
the total complexity for computing the eigendecomposition (\ref{eq:eig}) can
be counted following \cite[Section 3.1]{hsseig} and is $O(r^{2}n\log^{2}n)$.
(There is an erratum for \cite{hsseig} in the flop count since $r$ in equation
(3.1)\ of \cite[Section 3.1]{hsseig} should be $r^{2}$.) Note that the use of
all the new stability techniques here does not change the overall complexity.
Every local eigenmatrix $\hat{Q}_{i}$ is represented by a sequence of $r$
Cauchy-like matrices like in (\ref{eq:eigvecs}). Each such a Cauchy-like
matrix is stored with the aid of five vectors like in (\ref{eq:vec}). The
storage for $Q$ is then $O(rn\log n)$ and the cost to apply $Q$ or $Q^{T}$ to
a vector is $O(rn\log n)$ as in \cite{hsseig}.

\section{Numerical experiments\label{sec:tests}}

We then make a comprehensive test of the SuperDC eigensolver in terms of
different types of matrices and demonstrate its efficiency and accuracy.
SuperDC has been implemented in Matlab (available from
https://www.math.purdue.edu/\string~xiaj) and is compared with the highly
optimized Matlab \textsf{eig} function for computing the eigendecomposition.
We also show the significance of our stability techniques. The accuracy
measurements follow those in \cite{gu95,hsseig}:
\begin{align*}
\hspace{3cm}\gamma &  =\max\limits_{1\leq k\leq n}\frac{\Vert A\mathbf{q_{k}}-\lambda_{k}\mathbf{q}_{k}\Vert_{2}}{n\Vert A\Vert_{2}} &  &
\text{(residual),}\\
\delta &  =\frac{\sqrt{\sum_{k=1}^{n}(\lambda_{k}^{\ast}-\lambda_{k})^{2}}}{n\sqrt{\sum_{k=1}^{n}(\lambda_{k}^{\ast})^{2}}} &  &  \text{(relative
error),}\\
\theta &  =\max\limits_{1\leq k\leq n}\frac{\Vert Q^{T}{\mathbf{q}}_{k}-\mathbf{e}_{k}\Vert_{2}}{n} &  &  \text{(loss of orthogonality),}\end{align*}
where $\lambda_{k}^{\ast}$'s are eigenvalues from \textsf{eig} and are
considered as the exact results. The triangular FMM routine is developed based
on a code used in \cite{fmm1d}. The accuracy of each triangular FMM is set to
reach full machine precision so that it does not interfere with the
orthogonality of the eigenvectors.
The tests are performed with four 2.60GHz cores and 80GB memory on a node at a
cluster of Purdue RCAC. The usage of 80GB memory is just to accommodate the
need of \textsf{eig} for larger matrices.


\begin{example}
\label{ex1}We first consider a symmetric tridiagonal matrix $A$.
The classical divide-and-conquer eigensolver does not need tridiagonal
reduction and can be directly applied to $A$ with $O(n^{3})$ cost and
$O(n^{2})$ storage. For our SuperDC eigensolver, the HSS representation of $A$
can be explicitly written out without any extra cost and its HSS rank is $r=2$
\cite{phdthesis}. (The HSS structure does not rely on the actual nonzero
entries, which are $3$ on the main diagonal and $-1$ on the first
superdiagonal and subdiagonal. Other numbers such as random ones are also
tested with similar performance observed.) The size $n$ of $A$ in the test
ranges from $8192$ to $262,144$. In the HSS\ form, the leaf-level diagonal
block size is $2048$. We use $\tau=10^{-10}$ in the deflation criterion
(Remark \ref{rem:deflation}).
\end{example}

The timing (in seconds) of SuperDC and \textsf{eig} are reported in Figure
\ref{fig:ex1}(a). The storage for the eigenmatrix $Q$ (in terms of nonzeros)
is given in Figure \ref{fig:ex1}(b). The costs of SuperDC in terms of the
eigendecomposition flops and the flops to apply $Q$ to a vector\ are given in
Figure \ref{fig:ex1}(c). SuperDC achieves nearly linear complexity in all the
aspects (timing, flops, and storage), while \textsf{eig} exhibits a cubic
trend in timing and an obvious quadratic storage (which is just $n^{2}$ for
storing the dense $Q$). In fact, the flop count of SuperDC in Figure
\ref{fig:ex1}(c) shows a pattern even slightly better than $O(n\log^{2}n)$.
The timing is slightly off, likely due to the implementation.

\begin{figure}[th]
\centering
\subfigure[Eigendecomposition timing]{
\includegraphics[width=.31\textwidth]{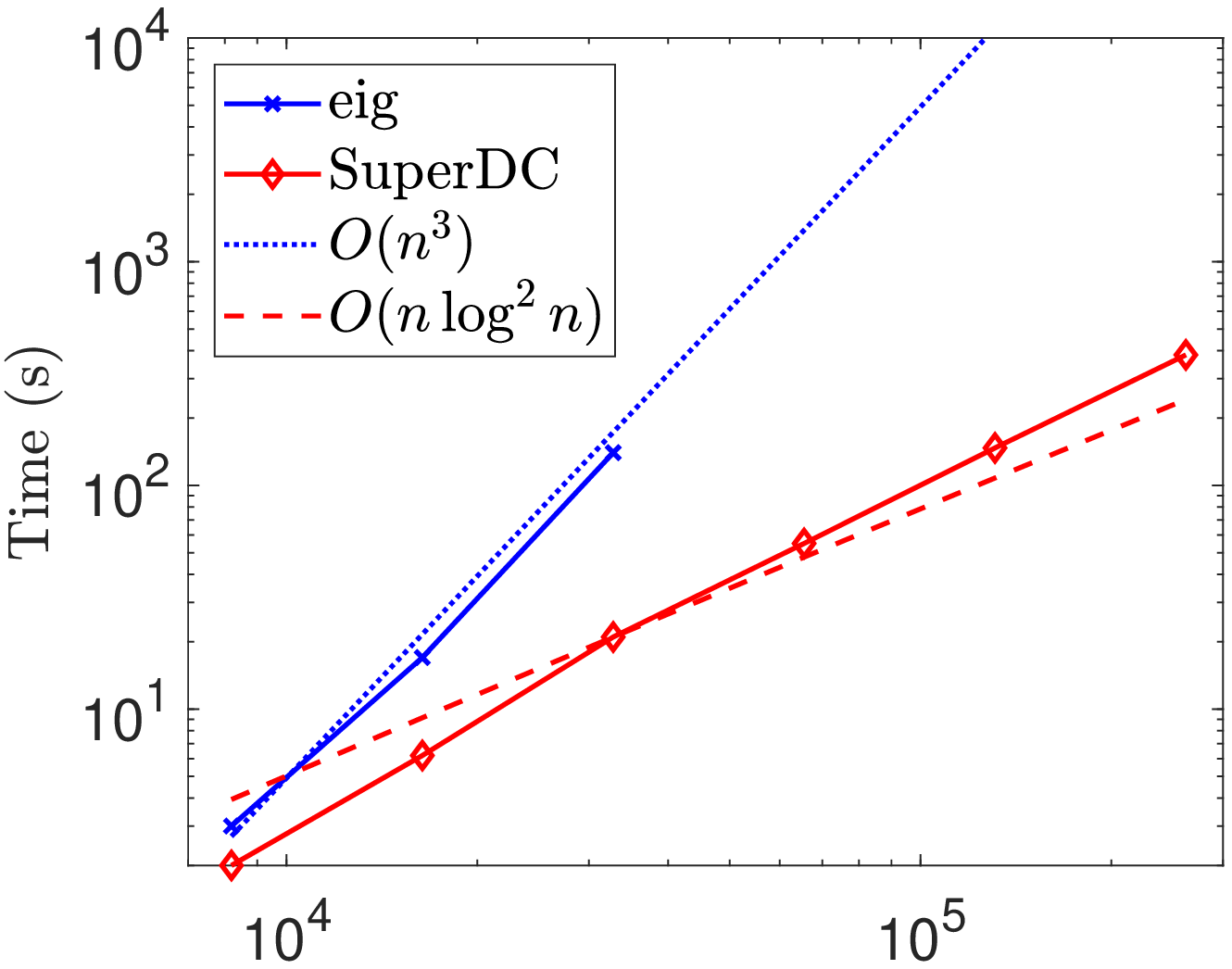}
} \subfigure[Storage]{
\includegraphics[width=.31\textwidth]{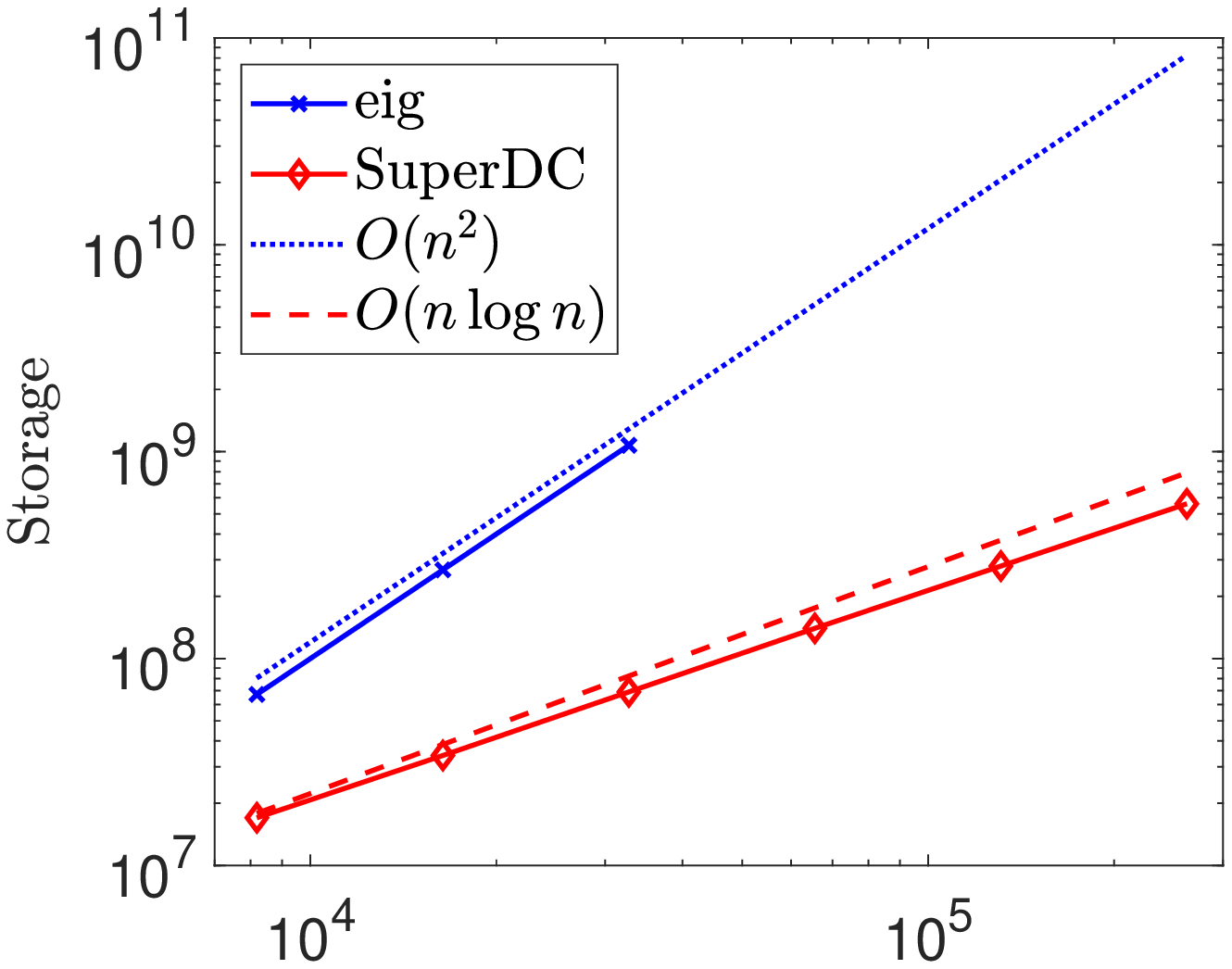}
} \subfigure[Flops of SuperDC]{
\includegraphics[width=.31\textwidth]{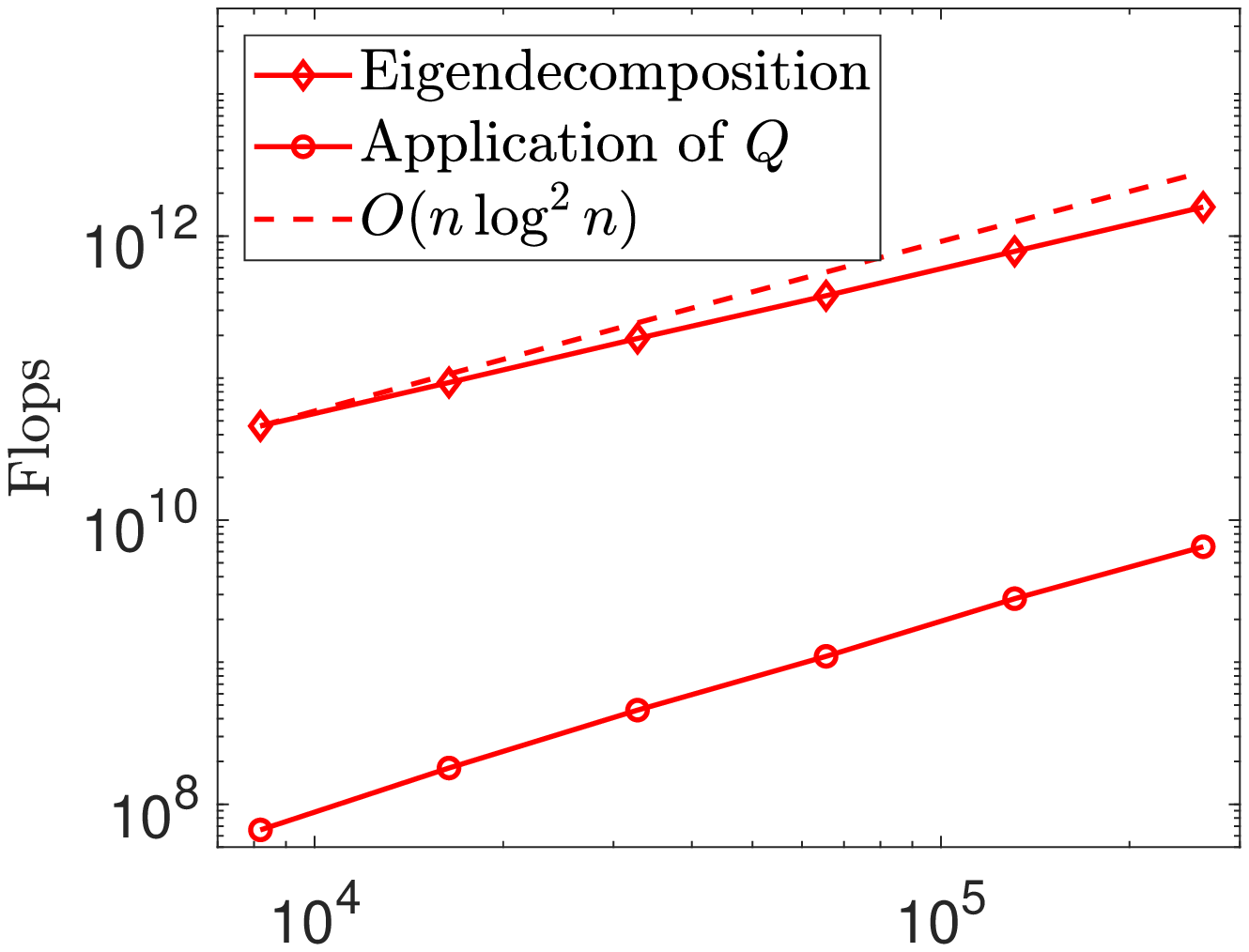}
}\caption{\emph{Example 1}. Timing and storage of SuperDC and \textsf{eig} and
flops of SuperDC.}\label{fig:ex1}\end{figure}

SuperDC is faster than \textsf{eig} for all the tested sizes. With $n=32,768$,
SuperDC is already over $6$ times faster than \textsf{eig} and takes only
about $6\%$ of the memory. (We also tested $n$ as small as $4096$ and SuperDC
already has lower storage and has comparable timing.) Note that \textsf{eig} runs
out of memory for larger $n$ due to the dense eigenmatrix, while SuperDC takes
much less memory and can reach much larger $n$.
For SuperDC, the timing is mostly for the conquering stage. For example, for
$n=262,144$, SuperDC takes $383.4$ seconds, where the dividing stage needs
just $16.1$ seconds. This also confirms that our strategy in Section
\ref{sec:divide} for reducing the ranks of low-rank updates is important since
it directly saves the cost in the conquering stage.

Table \ref{tb:tridiagaccuracy} shows the accuracy of SuperDC. The eigenvalues
are computed accurately and the loss of orthogonality $\theta$ is around
machine precision.

\begin{table}[h]
\caption{\emph{Example \ref{ex1}.} Accuracy of SuperDC, where some errors
($\delta$) are not reported since \textsf{eig} runs out of memory, and the
case $n=262,144$ is not shown since it takes too long to compute $\gamma$ and
$\theta$.}\label{tb:tridiagaccuracy}\centering\tabcolsep6pt\renewcommand{\arraystretch}{1.05}
\begin{tabular}
[c]{|c|c|c|c|c|c|}\hline
$n$ & $8,192$ & $16,384$ & $32,768$ & $65,536$ & $131,072$\\\hline
$\gamma$ & $1.9e-16$ & $8.8e-16$ & $5.2e-16$ & $3.0e-16$ & $1.5e-16$\\
$\delta$ & $1.6e-18$ & $8.0e-18$ & $2.9e-18$ &  & \\
$\theta$ & $6.4e-16$ & $2.3e-16$ & $1.9e-16$ & $2.1e-16$ & $1.8e-16$\\\hline
\end{tabular}
\end{table}

\begin{example}
\label{ex2}Next, we consider a symmetric matrix $A$ which is sparse and nearly
banded. That is, $A$ has a banded form with half bandwidth $5$ together with
some nonzero entries away from the band. The HSS\ form for $A$ can be
explicitly written out and has HSS rank $10$. The main diagonal entries are
equal to $3$ and the other entries in the band are equal to $-1$. The nonzero
entries away from the band are introduced by modifying some HSS generators for
the banded matrix constructed with the method in \cite{phdthesis}. In the
HSS\ form, the leaf-level diagonal block size is $2048$. We use $\tau
=10^{-10}$ in the deflation criterion (Remark \ref{rem:deflation}).

\end{example}

The entries away from the band break the banded structure of $A$. The
efficiency benefit of SuperDC becomes even more significant, as shown in
Figure \ref{fig:ex2}. At $n=32,678$, SuperDC is already about $11$ times
faster than \textsf{eig} and takes only about $7\%$ of the memory. Again,
\textsf{eig} runs out of memory when $n$ increases further, but SuperDC works
for much larger $n$ and demonstrates nearly linear complexity in the all aspects.

Table \ref{tb:bandedaccuracy} shows the accuracy of SuperDC. Similarly, high
accuracies are achieved.

\begin{figure}[th]
\centering
\subfigure[Eigendecomposition timing]{
\includegraphics[width=.31\textwidth]{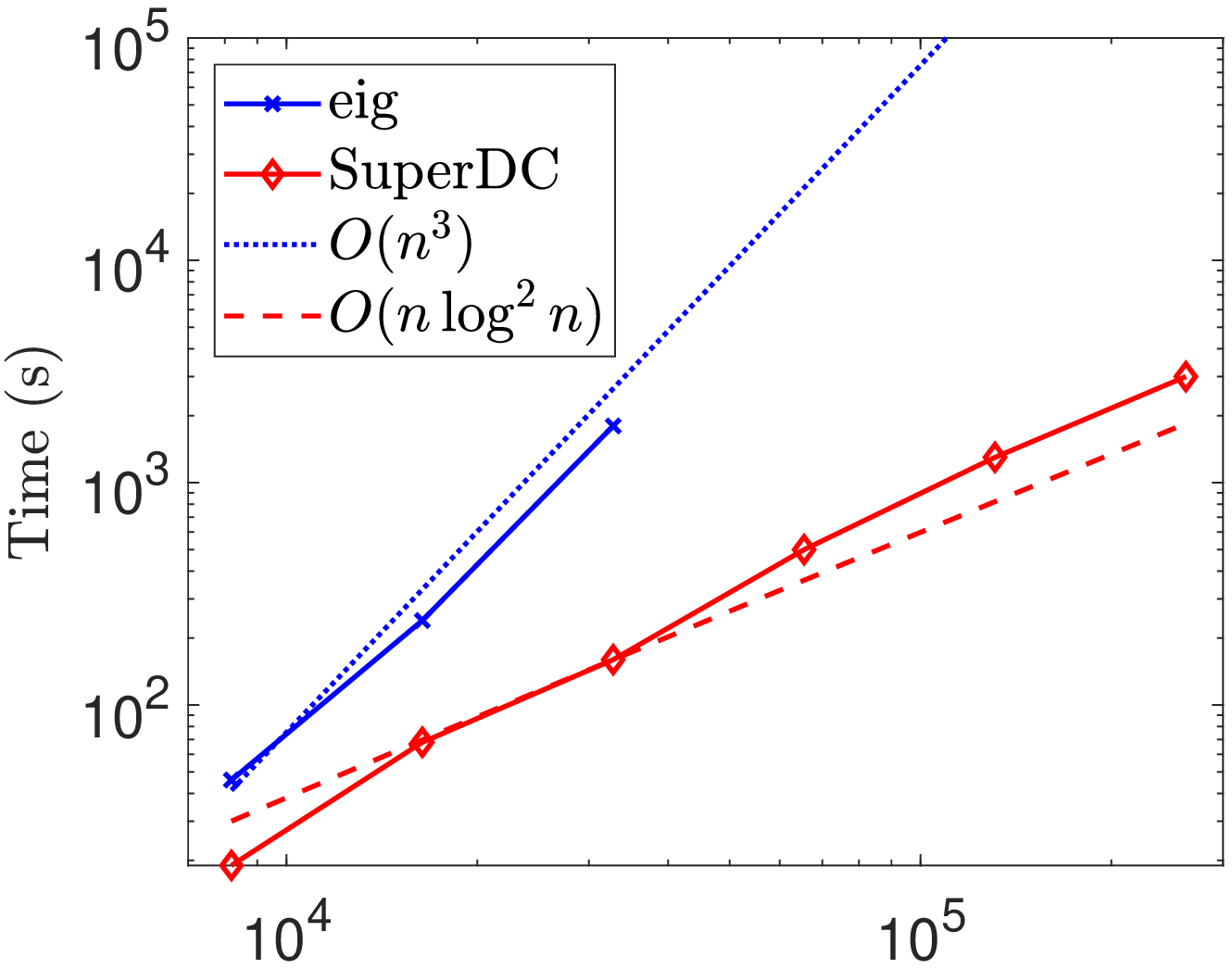}
} \subfigure[Storage]{
\includegraphics[width=.31\textwidth]{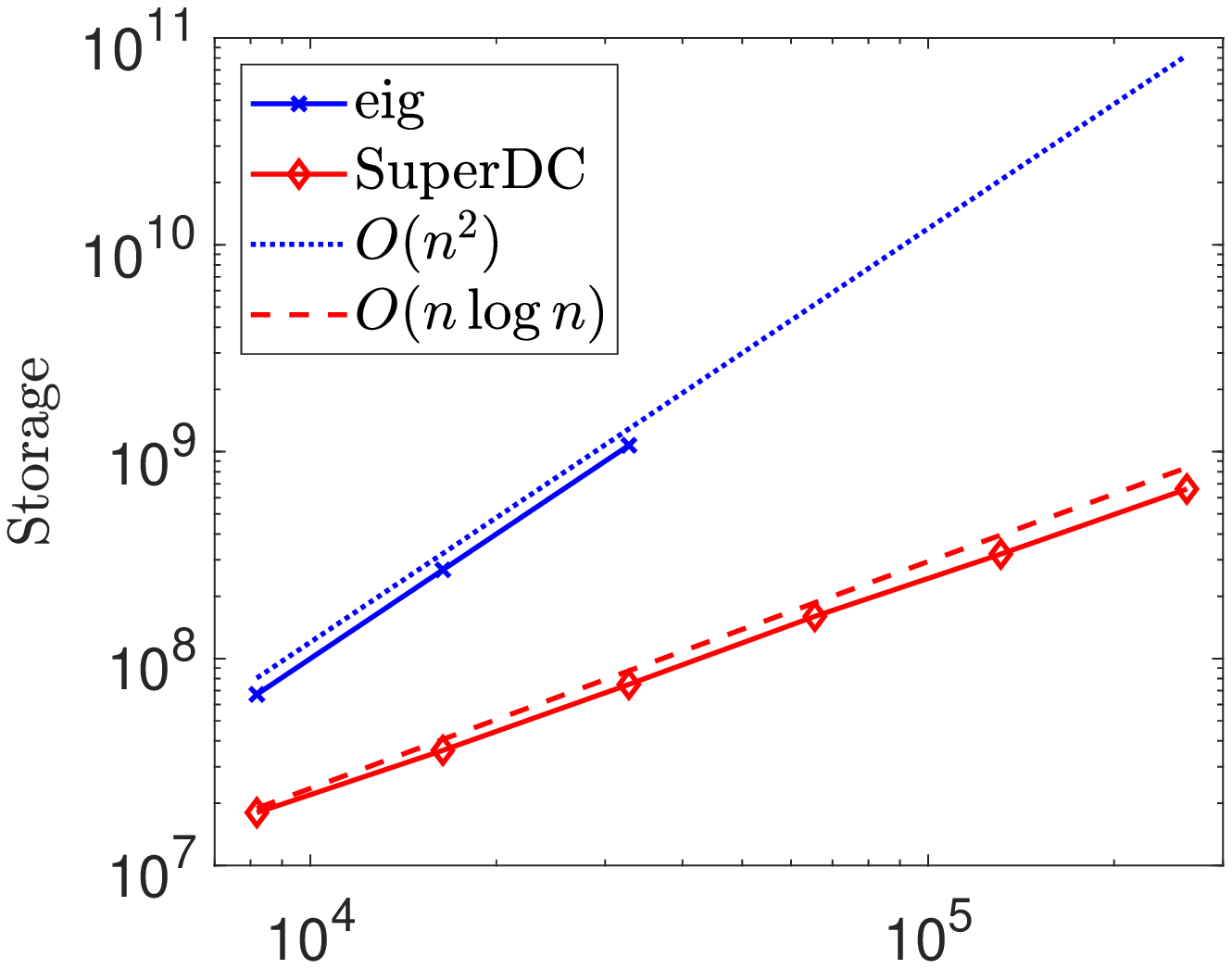}
} \subfigure[Flops of SuperDC]{
\includegraphics[width=.31\textwidth]{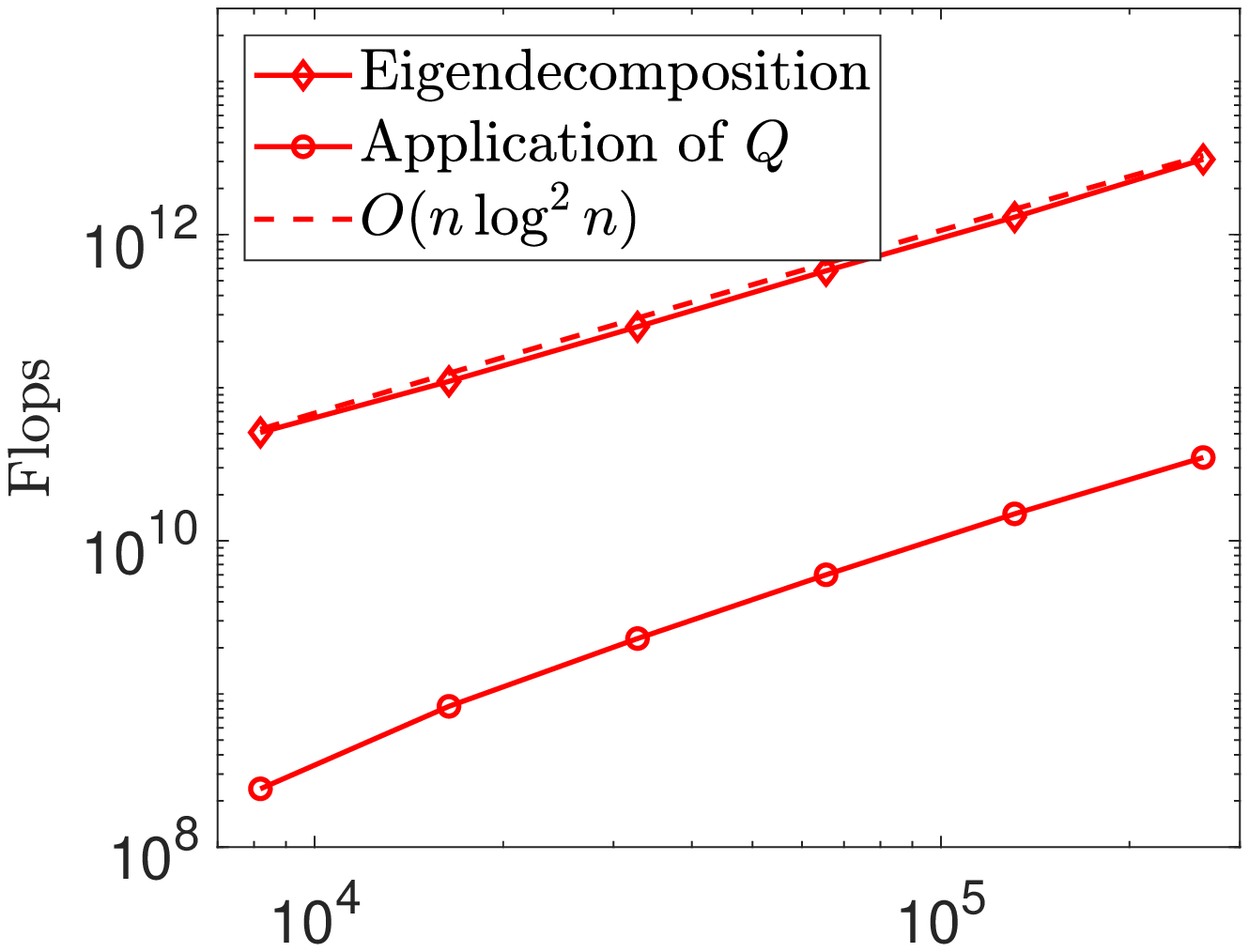}
}\caption{\emph{Example \ref{ex2}.} Timing and storage of SuperDC and
\textsf{eig} and flops of SuperDC.}\label{fig:ex2}\end{figure}

\begin{table}[h]
\caption{ \emph{Example \ref{ex2}.} Accuracy of SuperDC, where some errors
($\delta$) are not reported since \textsf{eig} runs out of memory, and the
case $n=262,144$ is not shown since it takes too long to compute $\gamma$ and
$\theta$.}\label{tb:bandedaccuracy}\centering\tabcolsep6pt\renewcommand{\arraystretch}{1.05}
\begin{tabular}
[c]{|c|c|c|c|c|c|}\hline
$n$ & $8,192$ & $16,384$ & $32,768$ & $65,536$ & $131,072$\\\hline
$\gamma$ & $6.5e-15$ & $6.8e-15$ & $1.2e-14$ & $1.1e-15$ & $1.1e-14$\\
$\delta$ & $1.4e-17$ & $6.5e-18$ & $8.1e-17$ &  & \\
$\theta$ & $1.8e-15$ & $2.8e-15$ & $2.4e-15$ & $4.8e-15$ & $2.3e-15$\\\hline
\end{tabular}
\end{table}

In addition, in order to demonstrate the importance of our local shifting
strategy, we have tested the eigensolver with triangular FMM\ accelerations applied to
the original secular equation instead of the shifted one. Other than the case
with $n=8192$, Matlab returns {\tt NaN} (not-a-number) for all the larger matrix
sizes due to cancellations. This confirms the risk of directly applying
FMM\ accelerations to the usual secular equation like in \cite{hsseig}.

\begin{example}
\label{ex3}Then consider a dense symmetric matrix $A$ which is a Toeplitz
matrix with its first row $\boldsymbol{\xi}=\begin{pmatrix}
\xi_{1} & \cdots & \xi_{n}\end{pmatrix}
$ given by\[
\xi_{1}=2\alpha,\quad\xi_{j}=\frac{\sin(2\alpha(j-1)\pi)}{(j-1)\pi
},\ j=2,3,\ldots,n,
\]
where $0<\alpha<1/2$. This is the so-called Prolate matrix that appears
frequently in signal processing. It is known to be extremely ill-conditioned
and has special spectral properties (see, e.g., \cite{var93}). In this
example, we set $\alpha=\frac{1}{4}$. It is known that any Toeplitz matrix can
be converted into a Cauchy-like matrix $\mathcal{C}$ which has small
off-diagonal numerical ranks \cite{toep,mar05,hsseig}. That is, $\mathcal{C}=\mathcal{F}A\mathcal{F}^{\ast}$, where $\mathcal{F}$ is the normalized
inverse DFT\ matrix. The eigendecomposition of $A$ can then be done via that
of $\mathcal{C}$. An HSS\ approximation to $\mathcal{C}$ may be quickly
constructed based on randomized methods in
\cite{parhssrsmf,mar11,mfhssrs,toeprs} coupled with fast Toeplitz
matrix-vector multiplications. The cost is nearly linear in $n$. Here, we use
a tolerance $10^{-10}$ in relevant compression steps, which is same as the
deflation tolerance $\tau$. In the HSS\ form, the leaf-level diagonal block
size is $2048$. SuperDC is applied to the resulting HSS\ form and compared
with \textsf{eig} applied to $A$. The size $n$ ranges from $4096$ to $65,536$.
\end{example}

In Figure \ref{fig:ex3}, the timing, storage, and flops are shown and are
consistent with the complexity estimates. The eigendecomposition with
SuperDC\ shows a dramatic efficiency advantage over \textsf{eig}. At
$n=32,768$, \textsf{eig} takes $1526.2$ seconds, while SuperDC only needs
$11.2$ seconds, which is a difference of about $136$ times. Also, the memory
saving is about $15$ times.

\begin{figure}[th]
\centering\subfigure[Eigendecomposition timing]{
\includegraphics[width=.31\textwidth]{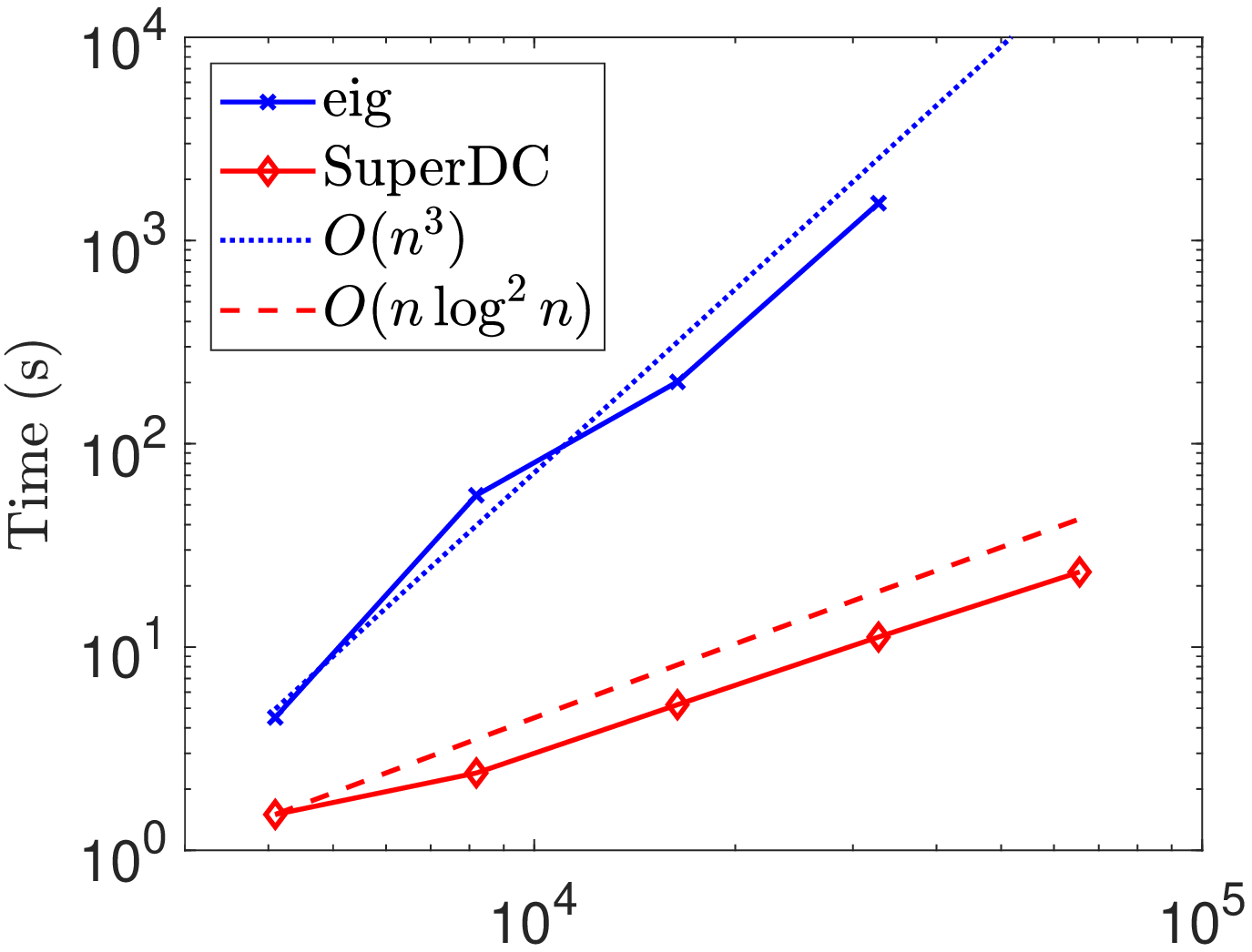}
} \subfigure[Storage]{
\includegraphics[width=.31\textwidth]{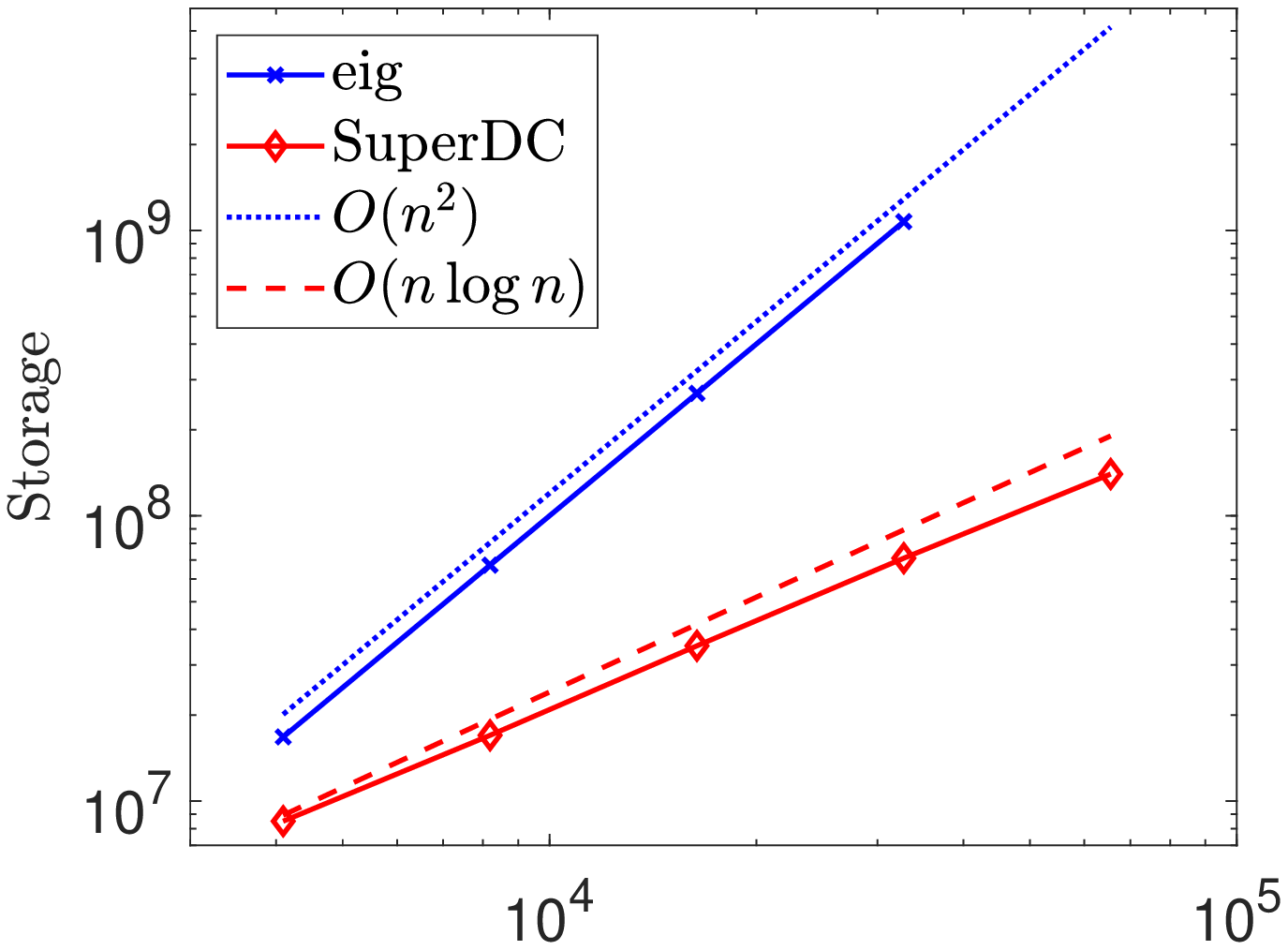}
} \subfigure[Flops of SuperDC]{
\includegraphics[width=.31\textwidth]{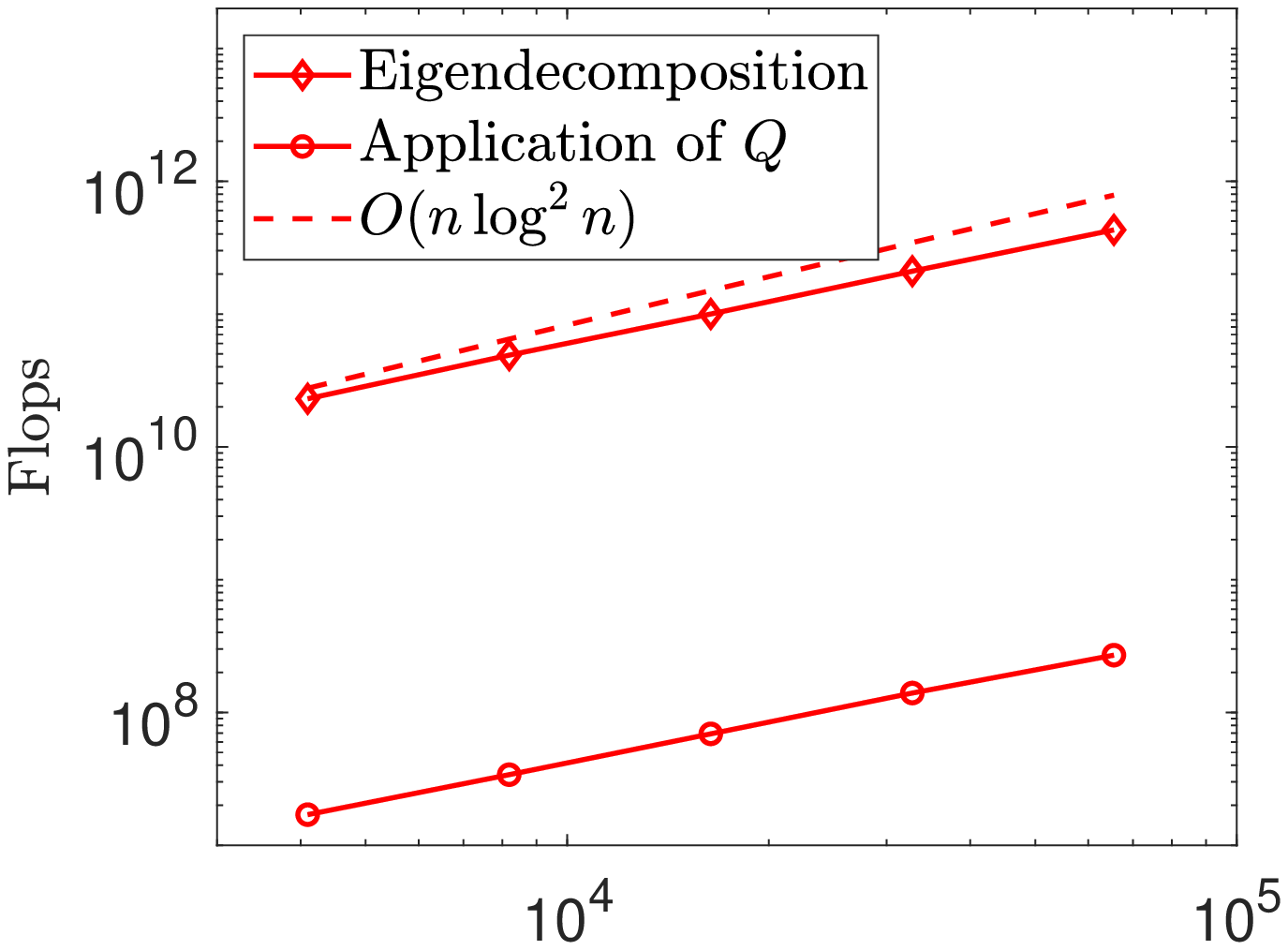}
}\caption{\emph{Example \ref{ex3}.} Timing and storage of SuperDC and
\textsf{eig} and flops of SuperDC.}\label{fig:ex3}\end{figure}

One thing we want to point out is that SuperDC has the theoretical complexity
$O(r^{2}n\log^{2}n)$, which may overestimate the actual cost. For example,
here $r$ is typically known to be $O(\log n)$. (This bound is based on
entrywise approximations, although a precise numerical rank may be slightly
higher \cite{mhs}.)\ One reason for the overestimate is that the flop count
does not take into consideration a levelwise rank pattern in \cite{hsscost}.
Another reason is our flexible deflation strategy in Remark
\ref{rem:deflation}. The matrices actually have highly clustered eigenvalues,
which further leads to high efficiency gain.


Despite the clustered eigenvalues, SuperDC still computes the
eigendecomposition accurately. See Table \ref{tb:toepaccuracy}.

\begin{table}[h]
\caption{\emph{Example \ref{ex3}.} Accuracy of SuperDC, where the error
($\delta$) for $n=65,536$ is not reported since \textsf{eig} runs out of
memory.}\label{tb:toepaccuracy}\centering\tabcolsep6pt\renewcommand{\arraystretch}{1.05}
\begin{tabular}
[c]{|c|c|c|c|c|c|}\hline
$n$ & $4,096$ & $8,192$ & $16,384$ & $32,768$ & $65,536$\\\hline
$\gamma$ & $1.5e-16$ & $6.7e-14$ & $2.3e-17$ & $1.3e-17$ & $3.3e-18$\\
$\delta$ & $7.8e-15$ & $1.7e-15$ & $2.3e-14$ & $3.7e-14$ & \\
$\theta$ & $3.4e-17$ & $9.3e-17$ & $4.6e-17$ & $2.3e-17$ & $1.2e-17$\\\hline
\end{tabular}
\end{table}

\begin{example}
\label{ex4}Our last example is a discretized kernel matrix $A$ in \cite{cha06}
which is the evaluation of the function $\sqrt{|s-t|}$ at the Chebyshev points
$\cos(\frac{2i-1}{2n}\pi),i=1,2,\ldots,n$.
The HSS\ construction may be based on direct off-diagonal compression or
efficient analytical methods like in \cite{kercompr}. We use an existing
routine based on the former one for simplicity. To show the flexibility of
accuracy controls, we aim for moderate accuracy in this test by using a
compression tolerance $10^{-6}$ in the HSS\ construction, which is same as the
deflation tolerance $\tau$.
\end{example}

For this example, we can observe similar complexity results as in the previous
examples. See Figure \ref{fig:ex4}, where we still set the leaf-level diagonal
block size to be $2048$ in the HSS approximations. With the larger tolerance
than in the previous examples, we still achieve reasonable eigenvalue errors
and residuals as in Table \ref{ex4:kernelaccuracy}. The loss of orthogonality
is still close to machine precision. Thus for the remaining discussions, we
focus on some stability advantages of SuperDC.

\begin{figure}[th]
\centering\subfigure[Eigendecomposition timing]{
\includegraphics[width=.31\textwidth]{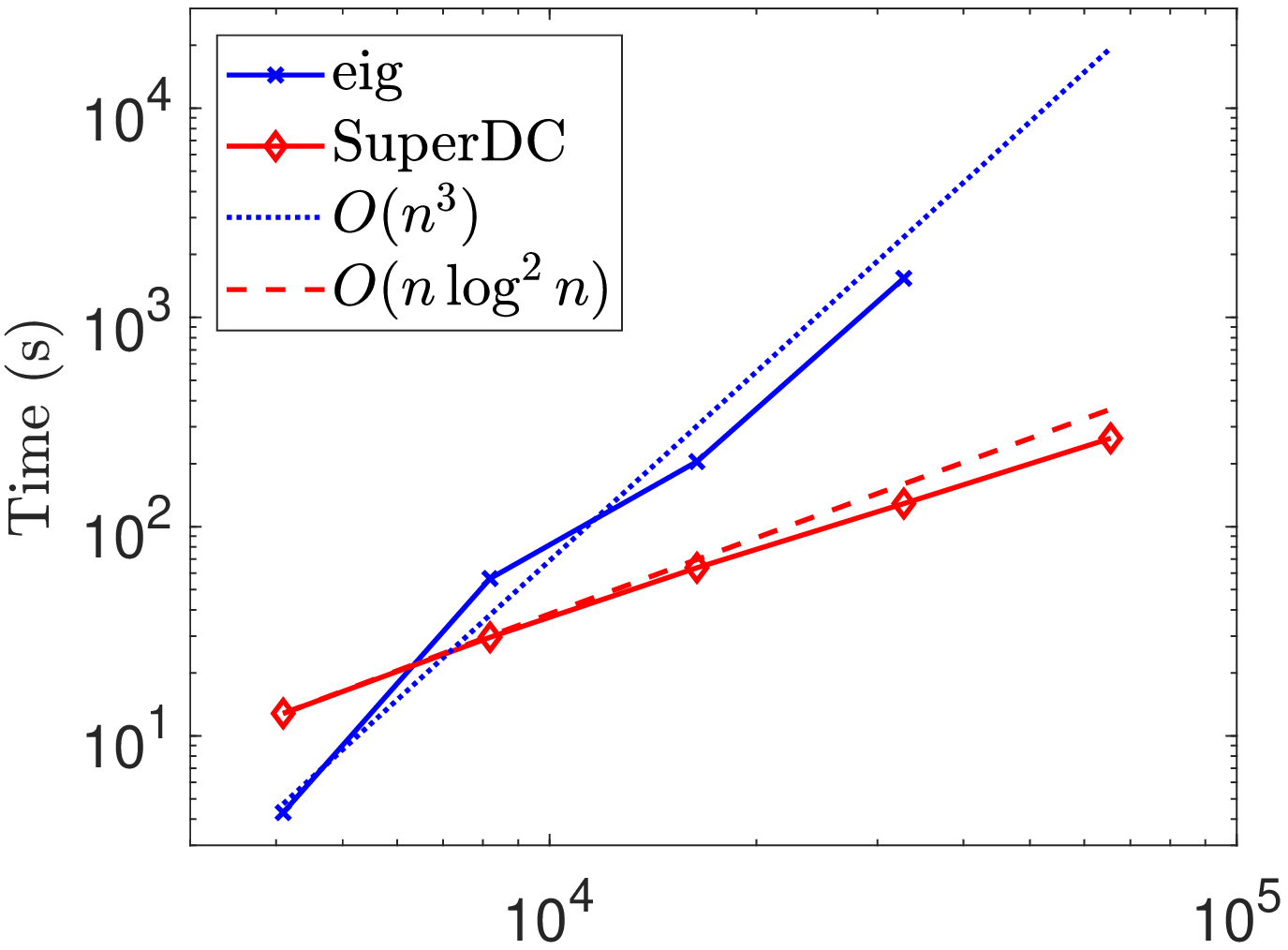}
} \subfigure[Storage]{
\includegraphics[width=.31\textwidth]{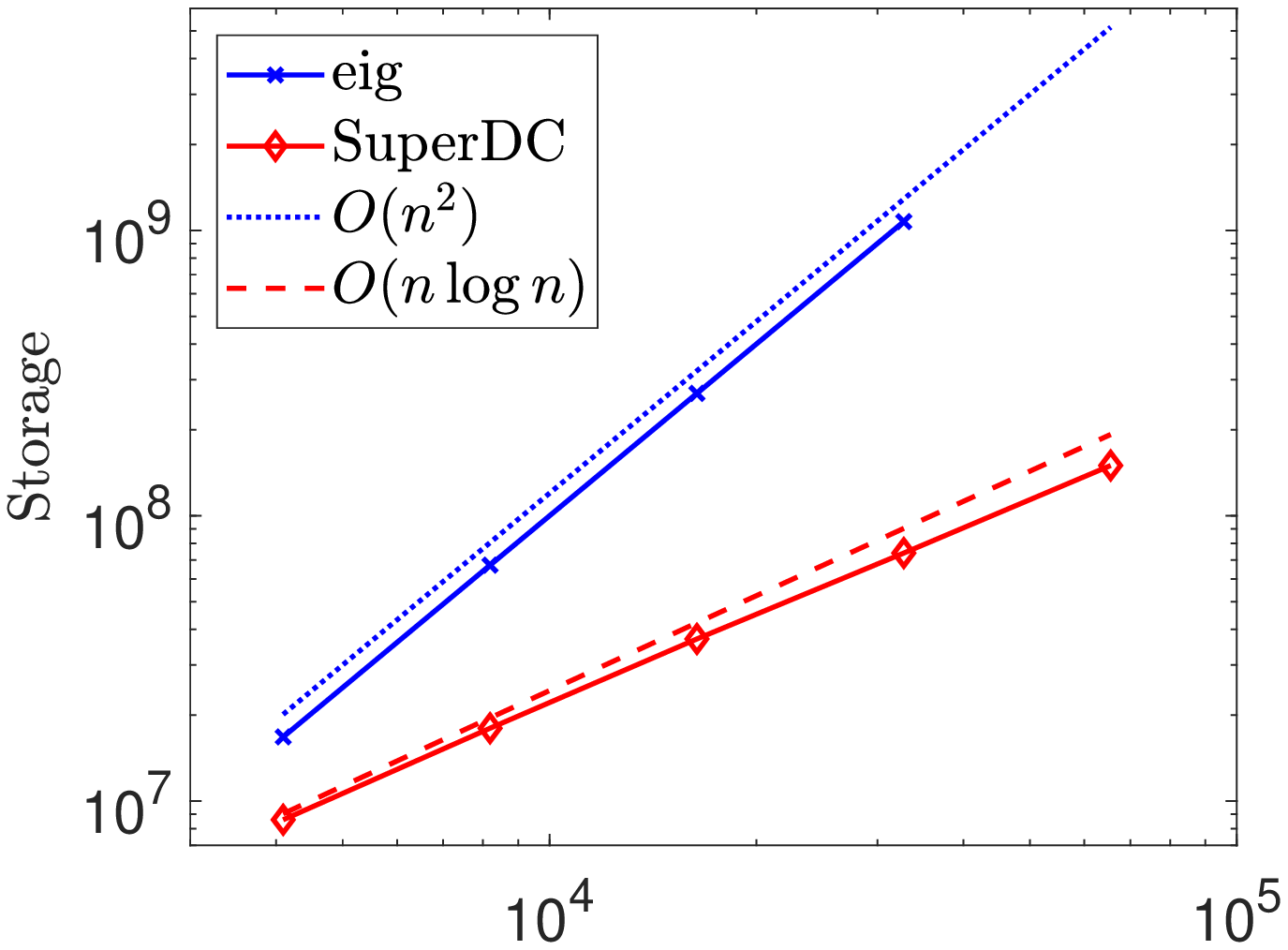}
}\subfigure[Flops of SuperDC]{
\includegraphics[width=.31\textwidth]{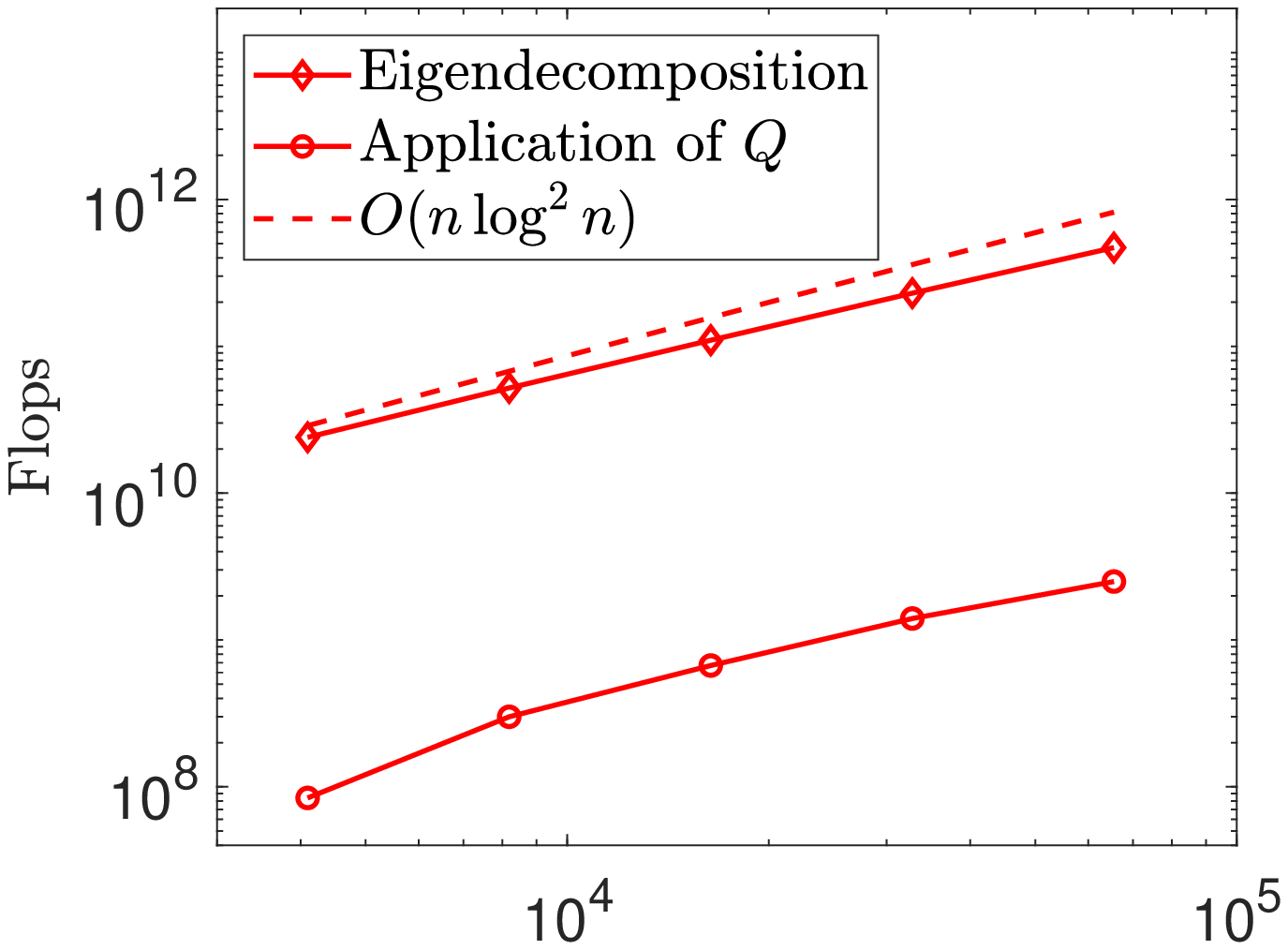}
}\caption{\emph{Example \ref{ex4}.} Timing and storage of SuperDC and
\textsf{eig} and flops of SuperDC.}\label{fig:ex4}\end{figure}

\begin{table}[h]
\caption{\emph{Example \ref{ex4}.} Accuracy of SuperDC, where the error
($\delta$) for $n=65,536$ is not reported since \textsf{eig} runs out of
memory.}\label{ex4:kernelaccuracy}\centering\tabcolsep6pt\renewcommand{\arraystretch}{1.05}
\begin{tabular}
[c]{|c|c|c|c|c|c|}\hline
$n$ & $4,096$ & $8,192$ & $16,384$ & $32,768$ & $65,536$\\\hline
$\gamma$ & $1.2e-10$ & $9.9e-11$ & $1.7e-10$ & $1.5e-10$ & $9.1e-11$\\
$\delta$ & $1.6e-11$ & $2.5e-11$ & $2.1e-11$ & $1.1e-11$ & \\
$\theta$ & $1.6e-15$ & $8.0e-15$ & $2.7e-15$ & $3.4e-15$ & $3.4e-15$\\\hline
\end{tabular}
\end{table}

Earlier in Example \ref{ex2}, it is shown that local shifting helps avoid
cancellations so that FMM\ accelerations can be applied reliably. In fact,
even if there is no cancellation in the original secular equation solution,
our local shifting strategy (for triangular FMM-accelerated solution of the
shifted secular equation) can further greatly benefit the convergence. To
illustrate this, we perform the following count. Suppose $r$ secular equations
are solved due to $r$ rank-$1$ updates associated with the root node of the
HSS\ tree $\mathcal{T}$. Let $\mu_{j}$ be the percentage of eigenvalues that
have \emph{not} converged after $5$ iterations during modified Newton's
solution of the secular equation associated with the $j$th rank-one update,
and let $\mu=\max_{1\leq j\leq r}\mu_{j}$. Table \ref{tab:ex4perc} reports
this maximum percentage $\mu$ with varying $n$. With local shifting, a vast
majority of those eigenvalues (about $99\%$ or more) converges within $5$
iterations. This is significantly better than the case without local shifting
(i.e., when the original secular equation is solved with FMM accelerations).

\begin{table}[h]
\caption{Maximum percentage ($\mu$) of eigenvalues not converged within $5$
iterations for solving the $r$ secular equations associated with
$\operatorname{root}(\mathcal{T})$.}\label{tab:ex4perc}\centering\tabcolsep6pt\renewcommand{\arraystretch}{1.05}
\begin{tabular}
[c]{|c|c|c|c|c|c|}\hline
$n$ & $4,096$ & $8,192$ & $16,384$ & $32,768$ & $65,536$\\\hline
With local shifting & $1.03\%$ & $0.75\%$ & $0.91\%$ & $0.89\%$ & $0.84\%$\\
Without local shifting & $69.1\%$ & $65.2\%$ & $65.0\%$ & $60.1\%$ &
$53.7\%$\\\hline
\end{tabular}
\end{table}

We would then also like to demonstrate the advantage of our stable dividing
strategy in Section \ref{sec:divide} as compared with the original one in
\cite{hsseig}. Following Propositions \ref{prop:b0} and \ref{prop:b}, we show
the norm growth of the $B,D$ generators after the dividing stage. For the
initial $B,D$ generators of the original HSS\ form, let $\tilde{B},\tilde{D}$
denote the updated generators after the entire dividing stage is finished.
Then let
\[
\rho_{B}=\max\limits_{i<\operatorname{root}(\mathcal{T})}\!\Vert B_{i}\Vert_{2},\quad\rho_{D}=\max\limits_{i\text{: leaf}}\Vert D_{i}\Vert_{2},\quad\rho_{\tilde{B}}=\max\limits_{i<\operatorname{root}(\mathcal{T})}\!\Vert\tilde{B}_{i}\Vert_{2},\quad\rho_{\tilde{D}}=\max\limits_{i\text{:
leaf}}\Vert\tilde{D}_{i}\Vert_{2}.
\]
In order to better show the norm growth after multilevel dividing, we set the
leaf-level diagonal block size to be $256$ here so as to have more levels. For
each $n$, Table \ref{ex4:norm} shows the number of levels in the HSS
approximation. When $n$ increases, the HSS tree $\mathcal{T}$ grows deeper.
Table \ref{ex4:norm} shows that $\Vert A\Vert_{2}$ and $\rho(B)$ grow roughly
linearly with $n$. However, $\rho_{\tilde{B}}$ and $\rho_{\tilde{D}}$ grow
exponentially with the original dividing stage in \cite{hsseig}, as predicted
by Proposition \ref{prop:b0}. This poses a stability risk. When $n$ grows
beyond a certain size, overflow happens.
(Note that $\rho_{\tilde{D}}$ has a larger magnitude than $\rho_{\tilde{B}}$,
which is consistent with Proposition \ref{prop:b0}.) In contrast, the growth
of $\rho_{\tilde{D}}$ and $\rho_{\tilde{B}}$ with our new dividing strategy is
much slower and roughly follows the growth pattern of $\rho(B)$, as predicted
by Proposition \ref{prop:b}. Accordingly, our algorithm can handle much larger
$n$ much more reliably. \begin{table}[h]
\caption{\emph{Example \ref{ex4}.} Norm growth of the $D,B$ generators after
the dividing stage, where $\infty$ means overflow.}\label{ex4:norm}\centering\tabcolsep6pt\renewcommand{\arraystretch}{1.05}
\begin{tabular}
[c]{|c|c|c|c|c|c|c|}\hline
\multicolumn{2}{|c|}{$n$} & $4,096$ & $8,192$ & $16,384$ & $32,768$ &
$65,536$\\\hline
\multicolumn{2}{|c|}{Number of levels} & $5$ & $6$ & $7$ & $8$ & $9$\\
\multicolumn{2}{|c|}{$\Vert A\Vert_{2}$} & $3.4e03$ & $6.8e03$ & $1.4e04$ &
$2.7e04$ & $5.4e04$\\\hline\hline
\multirow{2}{*}{Initial} & $\rho_{B}$ & $2.3e03$ & $4.6e03$ & $9.2e03$ &
$1.8e04$ & $3.7e04$\\
& $\rho_{D}$ & $6.1e01$ & $4.3e01$ & $3.1e01$ & $2.2e01$ & $1.5e01$\\\hline
After the original & $\rho_{\tilde{B}}$ & $5.5e24$ & $9.9e53$ & $2.1e117$ &
$4.2e253$ & $\infty$\\
dividing strategy & $\rho_{\tilde{D}}$ & $3.0e49$ & $9.9e107$ & $4.5e234$ &
$\infty$ & $\infty$\\\hline
After the new & $\rho_{\tilde{B}}$ & $2.3e03$ & $4.6e03$ & $9.2e03$ & $2.1e04$
& $5.5e04$\\
dividing strategy & $\rho_{\tilde{D}}$ & $4.7e03$ & $1.2e04$ & $3.4e04$ &
$8.6e04$ & $2.4e05$\\\hline
\end{tabular}
\end{table}

\section{Conclusions\label{sec:concl}}

In this work, we have designed the SuperDC eigensolver that is both superfast
and stable. It significantly improves the original divide-and-conquer
algorithm in \cite{hsseig} in both the stability and the algorithm design. A
series of stability techniques is built into the different stages of the
algorithm. In particular, we avoid an exponential norm growth risk in the
dividing stage via a balancing strategy and are further able to combine FMM
accelerations with several key stability safeguards that have been used in
practical divide-and-conquer algorithms. We also give a variety of algorithm
designs and structure studies that have been missing or unclear in
\cite{hsseig}. The comprehensive numerical tests confirm the nearly linear
complexity and much higher efficiency than the Matlab \textsf{eig} function
for the eigendecomposition of different types of HSS\ matrices. Nice accuracy
and eigenvector orthogonality have been observed. Comparisons also illustrate
the benefits of our stability techniques.

The SuperDC eigensolver makes it feasible to use full eigendecompositions to
solve various challenging numerical problems as mentioned at the beginning of
the paper. A list of applications is expected to be included in
\cite{superdcapp}. In addition, we expect that the novel local shifting
strategy and triangular FMM\ accelerations are also useful for other
FMM-related matrix computations when stability and accuracy are crucial. In
our future work, we plan to provide a high-performance parallel
implementation, which will extend the applicability of the algorithm to
large-scale numerical computations.

\newpage\setcounter{page}{1}

\headers{Supplementary materials}{Xiaofeng Ou and Jianlin Xia}
\thispagestyle{plain}

\renewcommand{\thealgorithm}{\arabic{algorithm}}

\begin{center}
{\bf\uppercase{Supplementary materials:\\ List of major algorithms}}
\end{center}
\bigskip
\noindent Title of paper: {\em SuperDC: Stable superfast divide-and-conquer eigenvalue decomposition}\\[2mm]
\noindent Authors: Xiaofeng Ou and Jianlin Xia

\bigskip
These supplementary materials are pseudocodes that can help better understand the major algorithms in the paper.
\begin{itemize}
\item Algorithm \ref{dividalg}: the HSS dividing stage.

\item Algorithm \ref{alg:secularalg}: solving the secular equation for the
eigenvalues with triangular FMM accelerations and local shifting.

\item Algorithm \ref{alg:conqueralg}: the conquering stage for producing the eigendecomposition.

\item Algorithm \ref{alg:superdcmv}: application of the a local eigenmatrix
$Q_{i}$ or its transpose to a vector. This is used in Algorithm
\ref{alg:conqueralg} and also can be used to apply the global eigenmatrix $Q$
or its transpose to a vector when $i=\operatorname{root}(\mathcal{T})$.


\end{itemize}

For notational
convenience, we use $r$ to represent the column sizes of all $Z_{i}$ matrices
in the pseudocodes. $\mathcal{T}_{i}$ is also used to denote the subtree of
$\mathcal{T}$ rooted at node $i\in\mathcal{T}$. $Z(:,j)$ means the $j$-th
column of $Z$.

The following utility routines are used in the algorithms. To save space, we
are not showing pseudocodes for these routines.

\begin{itemize}
\item $\mathsf{updhss}(D_{i},U_{i},H)$: for an HSS block $D_{i}$ corresponding
to the subtree $\mathcal{T}_{i}$, update its $D,B$ generators to get those of
$D_{i}-U_{i}HU_{i}^{T}$ using Lemma \ref{lem:genupd}.

\item $\mathsf{trifmm}(\mathbf{d},\mathbf{x},\boldsymbol{y},\mathbf{w},\kappa)$: compute a matrix-vector product $K\mathbf{w}$ with the triangular
FMM and local shifting as in Sections \ref{subsub:trifmm} and
\ref{sub:localshift}, where $K=(\kappa(d_{i},x_{j}))_{d_{i}\in\mathbf{d},x_{j}\in\mathbf{x}}$ is a kernel matrix and $\boldsymbol{y}$ is the gap
vector (for accurately evaluating $\mathbf{x}-\mathbf{d}$). Note that the
triangular FMM\ is used to multiply the lower triangular part of $K$ with
$\mathbf{w}$ and the strictly upper triangular part of $K$ with $\mathbf{w}$
and the final result is the sum of the two products.

\item $\mathsf{mnewton}(\boldsymbol{\psi},\boldsymbol{\phi},\boldsymbol{\psi
}^{\prime},\boldsymbol{\phi}^{\prime})$: use the modified Newton's method to
compute corrections to the current approximate gap as in (\ref{eq:newupdate}),
where $\boldsymbol{\psi},\boldsymbol{\phi},\boldsymbol{\psi}^{\prime
},\boldsymbol{\phi}^{\prime}$ look like (\ref{eq:psiphi}) and
(\ref{eq:psiphi2}).

\item $\mathsf{iniguess}(\mathbf{d},\mathbf{w})$: compute the initial guess as
in \cite{li93} for the solution of the secular equation (\ref{eq:seceq}).

\item $\mathsf{deflate}(\mathbf{d},\mathbf{v},\tau)$: apply deflation with the
criterion in Remark \ref{rem:deflation}.
\end{itemize}

\begin{algorithm}
[ptbh]\caption{SuperDC dividing stage}\label{dividalg}

\begin{algorithmic}
[1]\Procedure{\sf divide}{$\{D_i\}_{i\in \mathcal{T}},\{U_i\}_{i\in
\mathcal{T}},\{R_i\}_{i\in \mathcal{T}},\{B_i\}_{i\in \mathcal{T}} $}

\For{node $i = \text{root}(\mathcal{T}), \ldots, 1$}\Comment{Dividing
$D_{i}$ in a top-down traversal}

\If{$i$ is a non-leaf node}

\If{$\operatorname{colsize}(B_{c_1})\le \operatorname{rowsize}(B_{c_1})$}\algorithmiccomment{$c_1,c_2$: children of $i$}

\State$D_{c_{1}}\leftarrow\mathsf{updhss}(D_{c_{1}},U_{c_{1}},\frac{1}{\Vert
B_{c_{1}}\Vert_{2}}B_{c_{1}}B_{c_{1}}^{T})$\algorithmiccomment{Update generators of $D_{c_{1}}$\newline\mbox{}\hfill to get those of $D_{c_{1}}-\frac{1}{\Vert B_{c_{1}}\Vert_{2}}U_{c_{1}}B_{c_{1}}B_{c_{1}}^{T}U_{c_{1}}^{T}$ like in Lemma
\ref{lem:genupd}}

\State$D_{c_{2}}\leftarrow\mathsf{updhss}(D_{c_{2}},U_{c_{2}},\Vert B_{c_{1}}\Vert_{2}I)$\algorithmiccomment{Update generators of $D_{c_{2}}$\newline\mbox{}\hfill to get those of $D_{c_{2}}-\Vert B_{c_{1}}\Vert_{2}U_{c_{2}}U_{c_{2}}^{T}$ like in Lemma
\ref{lem:genupd}}

\Else

\State$D_{c_{1}}\leftarrow\mathsf{updhss}(D_{c_{1}},U_{c_{1}},\Vert B_{c_{1}}\Vert_{2}I)$\algorithmiccomment{Update generators of $D_{c_{1}}$\newline\mbox{}\hfill to get those of $D_{c_{1}}-\Vert B_{c_{1}}\Vert_{2}U_{c_{1}}U_{c_{1}}^{T}$ like in Lemma
\ref{lem:genupd}}

\State$D_{c_{2}}\leftarrow\mathsf{updhss}(D_{c_{2}},U_{c_{2}},\frac{1}{\Vert
B_{c_{1}}\Vert_{2}}B_{c_{1}}^{T}B_{c_{1}})$\algorithmiccomment{Update generators of $D_{c_{2}}$\newline\mbox{}\hfill to get those of $D_{c_{2}}-\frac{1}{\Vert B_{c_{1}}\Vert_{2}}U_{c_{2}}B_{c_{1}}^{T}B_{c_{1}}U_{c_{2}}^{T}$ like in Lemma
\ref{lem:genupd}}

\EndIf\EndIf\EndFor

\For{node $i = 1,\ldots,\text{root}(\mathcal{T})$}\Comment{Form $Z_i$
in a bottom-up traversal} \If{$i$ is a non-leaf node}

\If{$\operatorname{colsize}(B_{c_1})\le \operatorname{rowsize}(B_{c_1})$}\algorithmiccomment{$c_1,c_2$: children of $i$}

\State$Z_{i}\leftarrow\begin{pmatrix}
\frac{1}{\sqrt{\Vert B_{c_{1}}\Vert_{2}}}U_{c_{1}}B_{c_{1}}\\
\sqrt{\Vert B_{c_{1}}\Vert_{2}}U_{c_{2}}\end{pmatrix}
$ \algorithmiccomment{Local update $Z$ matrix like in (\ref{eq:divnew1})}
\Else

\State$Z_{i}\leftarrow\begin{pmatrix}
\sqrt{\Vert B_{c_{1}}\Vert_{2}}U_{c_{1}}\\
\frac{1}{\sqrt{\Vert B_{c_{1}}\Vert_{2}}}U_{c_{2}}B_{c_{1}}^{T}\end{pmatrix}
$ \algorithmiccomment{Local update $Z$ matrix like in (\ref{eq:divnew2})}
\EndIf

\If{$i\neq \text{root}(\mathcal{T})$} \State$U_{i}\leftarrow\begin{pmatrix}
U_{c_{1}}R_{c_{1}}\\
U_{c_{2}}R_{c_{2}}\end{pmatrix}
$ \algorithmiccomment{Assemble $U_i$ for parent node of $i$}
\EndIf\EndIf\EndFor

\State\textbf{return} updated generators $\{D_{i}\}_{i\in\mathcal{T}},\{B_{i}\}_{i\in\mathcal{T}},\{Z_{i}\}_{i\in\mathcal{T}}$

\EndProcedure

\end{algorithmic}
\end{algorithm}

\begin{algorithm}
[ptbh]\caption{Secular equation solution for eigenvalues (of
$\operatorname{diag}(\mathbf{d})+\mathbf{v}\mathbf{v}^T$)}

\begin{algorithmic}
[1]\Procedure{\sf secular}{$\mathbf{d},\mathbf{v}$}

\algorithmiccomment{Eigenvalue solution via the solution of the shifted secular equation (\ref{eq:shiftedeq})}

\State$\mathbf{w\leftarrow v\odot v}$

\State$\mathbf{y}^{(0)}\leftarrow\mathsf{iniguess}(\mathbf{d},\mathbf{w})$\algorithmiccomment{Computation of the initial guess as in \cite{li93}}
\State$\boldsymbol{x}^{(0)}\leftarrow\mathbf{y}^{(0)}+\mathbf{d}$

\For {$j=0,1,\ldots$}

\State$[\boldsymbol{\psi},\boldsymbol{\phi}]\leftarrow\mathsf{trifmm}(\mathbf{d},\mathbf{x}^{(j)},\boldsymbol{y}^{(j)},\mathbf{w},\frac{1}{s-t})$\algorithmiccomment{Computation of $\boldsymbol{\psi},\boldsymbol{\phi}$ in
(\ref{eq:psiphi})}

\State$[\boldsymbol{\psi}^{\prime},\boldsymbol{\phi}^{\prime}]\leftarrow
\mathsf{trifmm}(\mathbf{d},\mathbf{x}^{(j)},\boldsymbol{y}^{(j)},\mathbf{w},\frac{1}{(s-t)^{2}})$\algorithmiccomment{Computation of
$\boldsymbol{\psi}',\boldsymbol{\phi}'$ in (\ref{eq:psiphi2})}

\State$\mathbf{f}\leftarrow\mathbf{e}+\boldsymbol{\psi}+\boldsymbol{\phi}$

\If{$|\mathbf{f}|<cn(\mathbf{e}+|\boldsymbol\psi|+|\boldsymbol\phi|)\epsilon$}
\algorithmiccomment{Stopping criterion} \State\textbf{break}
\EndIf\State$\Delta\mathbf{x}^{(j)}\leftarrow\mathsf{mnewton}(\boldsymbol{\psi
},\boldsymbol{\phi},\boldsymbol{\psi}^{\prime},\boldsymbol{\phi}^{\prime})$

\algorithmiccomment{Computation of root update with modified Newton's method}

\State$\boldsymbol{y}^{(j+1)}\leftarrow\boldsymbol{y}^{(j)}+\Delta
\mathbf{x}^{(j)}$\algorithmiccomment{Updated gap approximation as in (\ref{eq:newupdate})}
\State$\mathbf{x}^{(j+1)}\leftarrow\boldsymbol{y}^{(j+1)}+\mathbf{d}$\algorithmiccomment{Updated eigenvalue approximation} \EndFor

\State$\boldsymbol{\lambda}\leftarrow\mathbf{x}^{(j)}$, $\boldsymbol{\eta
}\leftarrow\boldsymbol{y}^{(j)}$\algorithmiccomment{Eigenvalue and gap
upon convergence}

\State\textbf{return} $\boldsymbol{\lambda},\boldsymbol{\eta}$ \EndProcedure

\end{algorithmic}

\label{alg:secularalg}
\end{algorithm}

\begin{algorithm}
[ptbh]\caption{SuperDC conquering stage}

\begin{algorithmic}
[1]\Procedure{\sf conquer}{$\{D_i\}_{i\in \mathcal{T}},\{U_i\}_{i\in
\mathcal{T}},\{R_i\}_{i\in \mathcal{T}},\{B_i\}_{i\in
\mathcal{T}},\{Z_i\}_{i\in \mathcal{T}}, \tau$}

\algorithmiccomment{The $D_i,B_i$ generators have been updated in the
dividing stage}

\For{node $i = 1, \ldots, \operatorname{root}(\mathcal{T})$}\Comment{Conquering in a postordered traversal}

\If{$i$ is a leaf node}\algorithmiccomment{Leaf-level eigendecomposition}

\State$(\boldsymbol{\lambda}_{i},\hat{Q}_{i})\leftarrow\mathsf{eig}(D_{i})$\algorithmiccomment{Via Matlab {\sf eig} function}

\Else


\State$\begin{pmatrix}
{Z_{i,1}}\\
{Z_{i,2}}\end{pmatrix}
\leftarrow Z_{i}$\algorithmiccomment{Partitioning following the sizes of
$D_{c_1}$ and $D_{c_2}$}

\State${Z_{i,1}}\leftarrow\mathsf{superdcmv}(Q_{c_{1}},{Z_{i,1}},1)$\algorithmiccomment{$Q_{c_1}^T{Z_{i,1}}$}

\State${Z_{i,2}}\leftarrow\mathsf{superdcmv}(Q_{c_{2}},{Z_{i,2}},1)$\algorithmiccomment{$Q_{c_2}^T{Z_{i,2}}$}

\State$Z_{i}\leftarrow\begin{pmatrix}
{Z_{i,1}}\\
{Z_{i,2}}\end{pmatrix}
$\algorithmiccomment{$\hat{Z}_i$ like in (\ref{eq:zupdate})}

\State$[\boldsymbol{\lambda}_{i}^{(0)},P_{i}]\leftarrow\mathsf{sort}(\boldsymbol{\lambda}_{c_{1}},\boldsymbol{\lambda}_{c_{2}})$\algorithmiccomment{Ordering of all the diagonal
entries\newline\mbox{}\hfill of
$\boldsymbol{\lambda}_{c_1},\boldsymbol{\lambda}_{c_2}$ together, with $P_i$
the permutation matrix}


\For {$j=1,2\ldots,r$}\algorithmiccomment{$r=\operatorname{colsize}(Z_i)$}

\State$[\mathbf{d}_{i}^{(j)},Z_{i}(:,j)]\leftarrow\mathsf{deflate}(\boldsymbol{\lambda}_{i}^{(j-1)},Z_{i}(:,j),\tau)$\algorithmiccomment{Deflation (Remark \ref{rem:deflation})}

\State$[\boldsymbol{\lambda}_{i}^{(j)},\boldsymbol{\eta}_{i}^{(j)}]\leftarrow\mathsf{secular}(\mathbf{d}_{i}^{(j)},Z_{i}(:,j))$\algorithmiccomment{Secular equation solution}

\State$\mathbf{v}_{1}\leftarrow\mathsf{trifmm}(\mathbf{d}_{i}^{(j)},\boldsymbol{\lambda}_{i}^{(j)},\boldsymbol{\eta}_{i}^{(j)},\mathbf{e},\log|s-t|)$\algorithmiccomment{$G_1\mathbf{e}$ as needed in (\ref{eq:logv})}

\State$\mathbf{v}_{2}\leftarrow\mathsf{trifmm}(\mathbf{d}_{i}^{(j)},\mathbf{d}_{i}^{(j)},\mathbf{0},\mathbf{e},\log|s-t|)$\algorithmiccomment{$G_2\mathbf{e}$ as needed in (\ref{eq:logv})}

\State$\mathbf{\hat{v}}_{i}^{(j)}\leftarrow\exp{(}\frac{{\mathbf{v}_{1}-\mathbf{v}_{2}}}{2}{)}$\algorithmiccomment{L\"{o}wner's formula for
$\mathbf{\hat{v}}$ as in (\ref{eq:lowner})--(\ref{eq:logv})}

\State$\mathbf{b}_{i}^{(j)}\leftarrow$ $(\mathsf{trifmm}(\mathbf{d}_{i}^{(j)},\boldsymbol{\lambda}_{i}^{(j)},\boldsymbol{\eta}_{i}^{(j)},\mathbf{\hat{v}}_{i}^{(j)}\odot\mathbf{\hat{v}}_{i}^{(j)},\frac{1}{(s-t)^{2}}))^{-1/2}$

\algorithmiccomment{Normalization factor as in (\ref{eq:b})}

\State$\hat{Q}_{i}^{(j)}\leftarrow\{\mathbf{\hat{v}}_{i}^{(j)},\mathbf{b}_{i}^{(j)},\mathbf{d}_{i}^{(j)},\boldsymbol{\lambda}_{i}^{(j)},\boldsymbol{\eta}_{i}^{(j)}\}$\algorithmiccomment{Cauchy-like
structured\newline\mbox{}\hfill representation of the local eigenmatrix as
in (\ref{eq:eigvecs})}

\For{$k=j+1,j+2,\ldots,r$}\algorithmiccomment{Multiplication of
$(\hat{Q}_i^{(j)})^T$\newline\mbox{}\hfill to the remaining columns of $Z_i$ via
the steps as in (\ref{eq:qscale})}


\State$Z_{i}(:,k)\leftarrow\mathbf{\hat{v}}_{i}^{(j)}\odot Z_{i}(:,k)$

\State $Z_{i}(:,k) \leftarrow -\mathsf{trifmm}(\boldsymbol{\lambda}_{i}^{(j)},\mathbf{d}_{i}^{(j)},\boldsymbol{\eta}_{i}^{(j)},Z_{i}(:,k),\frac{1}{s-t})$

\algorithmiccomment{The negative sign
and the switch of \newline\mbox{}\hfill $\boldsymbol{\lambda}_i^{(j)}$ and $\mathbf{d}_i^{(j)}$
are due to the transpose}

\State$Z_{i}(:,k)\leftarrow\mathbf{b}_{i}^{(j)}\odot Z_{i}(:,k)$

\EndFor

\EndFor

\State$\boldsymbol{\lambda}_{i}\leftarrow\boldsymbol{\lambda}_{i}^{(r)}$\algorithmiccomment{Local eigenvalues associated with node $i$}

\EndIf



\EndFor

\State$\boldsymbol{\lambda}\leftarrow\boldsymbol{\lambda}_{\operatorname{root}(\mathcal{T})}$, $Q\leftarrow\{\{\hat{Q}_{i}^{(j)}\}_{j=1}^{r},P_{i}\}_{i\in\mathcal{T}}$\algorithmiccomment{Final eigenvalues\newline\mbox{}\hfill and
eigenmatrix $Q$ in (\ref{eq:eigmat}), with $\hat{Q}_{i}$ in (\ref{eq:ql})
given by $\prod_{j=1}^{r}\hat{Q}_{i}^{(j)}$}

\State\textbf{return} $\boldsymbol{\lambda}, Q$

\EndProcedure

\end{algorithmic}

\label{alg:conqueralg}
\end{algorithm}

\begin{algorithm}
[ptbh]\caption{SuperDC eigenmatrix-vector multiplication}

\begin{algorithmic}
[1]\Procedure{\sf superdcmv}{$Q_ i, \mathbf{x},\text{transpose}$}\algorithmiccomment{Application of a local eigenmatrix
$Q_i$\newline\mbox{}\hfill or its transpose to a vector $\mathbf{x}$,
depending on whether `transpose' is $0$ or $1$}

\State$i_{1}\leftarrow$ smallest descendant of $i$

\If{$\text{transpose} = 0$}\algorithmiccomment{$\mathbf{y} = Q_i\mathbf{x}$}

\State$\mathbf{y}_{i}\leftarrow\mathbf{x}$

\For {$k=i,i-1,\ldots,i_1$}\algorithmiccomment{Reverse postordered traversal
of $\mathcal{T}_i$}

\If{$k$ is leaf}

\State$\mathbf{y}_{k}\leftarrow Q_{k}\mathbf{y}_{k}$\algorithmiccomment{Dense ${Q}_k$ at the leaf level}

\Else

\For {$j=r,r-1,\ldots,1$}\algorithmiccomment{Multiplication of
$\hat{Q}_k^{(j)}$\newline\mbox{}\hfill via the steps like in
(\ref{eq:qscale})}

\State$\mathbf{y}_{k}\leftarrow\mathbf{b}_{k}^{(j)}\odot\mathbf{y}_{k}$

\State$\mathbf{y}_{k}\leftarrow\mathsf{trifmm}(\mathbf{d}_{k}^{(j)},\boldsymbol{\lambda}_{k}^{(j)},\boldsymbol{\eta}_{k}^{(j)},\mathbf{y}_{k},\frac{1}{s-t})$

\State$\mathbf{y}_{k}\leftarrow\mathbf{\hat{v}}_{k}^{(j)}\odot\mathbf{y}_{k}$

\EndFor

\State$\mathbf{y}_{k}\leftarrow P_{k}^{T}\mathbf{y}_{k}$\algorithmiccomment{Permutation like in (\ref{eq:dphateig})}

\State$\begin{pmatrix}
\mathbf{y}_{c_{1}}\\
\mathbf{y}{_{c_{2}}}\end{pmatrix}
\leftarrow\mathbf{y}_{k}$\algorithmiccomment{Partitioning following the
sizes of $Q_{c_1},Q_{c_2}$,\newline\mbox{}\hfill with $c_1,c_2$ the children
of $k$}

\EndIf

\EndFor

\Else\algorithmiccomment{$\mathbf{y} = Q_i^T\mathbf{x}$}

\State Partition $\mathbf{x}$ into $\mathbf{x}_{k}$ pieces following the
leaf-level $Q_{k}$ sizes

\For {$k=i_1,i_1+1,\ldots,i$}\algorithmiccomment{Postordered traversal of
$\mathcal{T}_i$}

\If{$k$ is leaf}

\State$\mathbf{y}_{k}\leftarrow Q_{k}^{T}\mathbf{x}_{k}$\algorithmiccomment{Dense ${Q}_k$ at the leaf level}

\Else

\State$\mathbf{y}_{k}\leftarrow\begin{pmatrix}
\mathbf{y}_{c_{1}}\\
\mathbf{y}{_{c_{2}}}\end{pmatrix}
$\algorithmiccomment{$c_1,c_2$: children of $k$}

\State$\mathbf{y}_{k}\leftarrow P_{k}\mathbf{y}_{k}$\algorithmiccomment{Permutation like in (\ref{eq:dphateig})}

\For {$j=1,2,\ldots,r$}\algorithmiccomment{Multiplication of
$(\hat{Q}_k^{(j)})^T$\newline\mbox{}\hfill via the steps like in
(\ref{eq:qscale})}

\State$\mathbf{y}_{k}\leftarrow\mathbf{\hat{v}}_{k}^{(j)}\odot\mathbf{y}_{k}$

\State$\mathbf{y}_{k}\leftarrow-\mathsf{trifmm}(\boldsymbol{\lambda}_{k}^{(j)},\mathbf{d}_{k}^{(j)},\boldsymbol{\eta}_{k}^{(j)},\mathbf{y}_{k},\frac{1}{s-t})$\algorithmiccomment{The negative sign\newline\mbox{}\hfill
and the switch of $\boldsymbol{\lambda}_k^{(j)}$ and $\mathbf{d}_k^{(j)}$
are due to the transpose}

\State$\mathbf{y}_{k}\leftarrow\mathbf{b}_{k}^{(j)}\odot\mathbf{y}_{k}$

\EndFor

\EndIf\EndFor\EndIf

\State\textbf{return} $\mathbf{y}$

\EndProcedure

\end{algorithmic}

\label{alg:superdcmv}
\end{algorithm}


\end{document}